\documentclass[11pt, leqno]{amsart}
\usepackage{amsmath,amssymb,amsthm}
\usepackage{graphicx}
\usepackage{subfigure}
\usepackage{enumerate}
\usepackage{fullpage}
\usepackage{url}
\usepackage{float}
\usepackage{xspace}
\usepackage{hyperref}
\usepackage{color}
\usepackage{thmtools, thm-restate}
\usepackage{etoolbox}
\usepackage[page]{appendix}

\newtheorem{theorem}{Theorem}

\newtheorem{corollary}{Corollary}
\newtheorem{lem}{Lemma}[section]
\newtheorem{proposition}{Proposition}

\newtheorem{remark}{Remark}

\title{Moments of isotropic measures and optimal projective codes}

\author{Alexey Glazyrin}
\address{School of Mathematical \& Statistical Sciences, The University of Texas Rio Grande Valley, Brownsville, TX 78520}
\email{alexey.glazyrin@utrgv.edu}
\thanks{}
\subjclass[2000]{Primary 52C17, 52C35, 05B40; Secondary 33C55, 41A05}
\keywords{isotropic measure, code, design, potential energy minimization, linear programming bounds}
\date{\today}

\begin{document}

\maketitle
\begin{abstract}

In this paper, we use the linear programming approach to find new upper bounds for the moments of isotropic measures. These bounds are then utilized for finding lower energy bounds for projective codes. We also show that the obtained energy bounds are sharp for several infinite families of codes.

\end{abstract}

\section{Introduction}

To introduce the problems that will be addressed, let $\mathcal{C} = \{x_1,\dots,x_M\}\subset \mathbb{S}^{d-1}$ be a subset of points on the sphere in $\mathbb{R}^d.$ We will call $\mathcal{C}$ a spherical $\varphi$-code if the angular distance between any two points of $\mathcal{C}$ is not greater than $\varphi.$ The general problem is to find the maximum cardinality of a $\varphi$-code in $\mathbb{S}^{d-1}$. The problem of finding the maximum $\pi/3$-code is known as \textit{the kissing number problem}. For $d=3$, the problem of finding the maximal $\varphi$ such that there is a $\varphi$-code with $n$ points for a given $n$, is the Tammes problem \cite{tam30}. Similar problems can be solved for projective and, more generally, for Grassmanian spaces. In a seminal paper \cite{con96}, Conway, Hardin, and Sloane investigated optimal projective and Grassmanian codes computationally and proved optimality of many codes.

One of the most important methods in this area is the linear programming approach. The method was discovered by Delsarte \cite{del72, del73} for the Hamming space and then extended to the spherical case \cite{del77}; and was generalized by Kabatyansky and Levenshtein \cite{kab78} who use it to prove the current best asymptotic upper bound on the density of sphere packings in Euclidean spaces, $2^{-(0.5990\ldots+o(1))n}$ (improved by a constant in \cite{coh14}). 

The analogue of the Delsarte bound for unit sphere packings in Euclidean spaces is known as the Cohn-Elkies method \cite{coh03} (see also the independent result of Gorbachev \cite{gor00}). Using this method, Viazovska, in the recent groundbreaking work \cite{via17}, proved that $E_8$ gives the densest sphere packing in dimension 8. It was established in \cite{coh09} that the Leech lattice gives the densest sphere packing among lattices in dimension 24 and, soon after Viazovska's breakthrough, it was shown in \cite{coh17a} that the Leech lattice also solves the general sphere packing problem. Recently, the same approach was used to show the universal optimality of $E_8$ and the Leech lattice \cite{coh19}.

For a discrete configuration of points one can define a complete energy of this configuration and study point allocations minimizing it either locally, or globally. This approach is extremely popular in various areas of mathematics and science (see \cite{fop12, why52, coh56, mel77, saf97, bal09} for numerous examples). A noteworthy example is Problem 7 from the list of Mathematical Problems for the Next Century by Smale \cite{sma98}. For the case of the circle, typically the optimal configuration is a regular polygon. For the two-dimensional sphere, even the case of 5 points is highly non-trivial \cite{sch13, sch16}.

For a compact or finite metric space $M$ with the distance function $d$, given a potential real function $f$, we define the total energy of a point set $X=\{x_1,\ldots,x_N\}$ by 

$$I_f(X)=\sum_{i\neq j} f(d^2(x_i,x_j)).$$

The problem consists in finding the smallest energy and all minimizing configurations of points. Note that the packing problem can be interpreted as a special case of the energy optimizing problem either by considering non-continuous potentials or as a limiting configuration for the potential functions $e^{-\Theta d^2(x_i,x_j)}$ with $\Theta\rightarrow\infty$.

The linear programming machinery developed for finding packing bounds can be extended to this general setup as well. This extension was done by Yudin who also used it to find sharp energy bounds for $d+1$ and $2d$ points in $\mathbb{S}^{d-1}$ in \cite{yud92}. It was later applied to the set of minimal vectors of $E_8$, the set of minimal vectors of the Leech lattice, and the regular icosahedron in \cite{kol94, and96, and97}. In \cite{kol97}, Kolushov and Yudin ask how generally their technique can be applied. This question was answered by Cohn and Kumar in \cite{coh07} who proved that \textit{sharp configurations} on spheres ($s$-distance sets which are also $(2s-1)$-designs) are global minimizers for all completely monotonic potentials $f$, i.e. such that $(-1)^k f^{(k)} (t^2)\geq 0$ for all $k\geq 0$ and all possible distances $t$.

In \cite{buk18}, Bukh and Cox developed new packing bounds for codes in real and complex projective spaces. They also showed that these bounds are sharp for certain families of codes. The key element of their result is the upper bound on the first moment of isotropic measures. In this paper, we generalize their bound for moments of isotropic measures using the Yudin-type linear programming approach. Following their ideas we also develop new bounds on energies of projective codes and show that these bounds are sharp for several infinite families of codes.

Here is the list of the main results obtained in this paper.

\begin{enumerate}

\item We find the general linear programming bound for a certain class of energies defined on isotropic measures in the spaces $\mathbb{R}^d$, $\mathbb{C}^d$, $\mathbb{H}^d$ including $q$-th moments of isotropic measures for $q\in [1,2]$ (Theorem \ref{thm:moments} and Corollaries \ref{cor:yudin}-\ref{cor:yudin-moments}). This bound is the direct analog of Delsarte-Yudin linear programming energy bounds for codes in two-point homogeneous spaces.

\item We use this general linear programming bound to prove that the $q$-th moment of an isotropic measure in $K^d$ ($K=\mathbb{R}$, $\mathbb{C}$, or $\mathbb{H}$) is not greater than $\beta^q + \frac {1-\beta^q} {M},$ where $\beta=\sqrt{\frac {1} {d+2(dim_{\mathbb{R}} K)^{-1}}}$, $M=d+ \frac {d^2-d} 2 dim_{\mathbb{R}} K$ (Theorem \ref{thm:q-measures}). This bound is sharp precisely for uniform distributions over maximal projective simplices that are known to exist for $d=2, 3, 7, 23$ when $K=\mathbb{R}$, $d=3$ when $K=\mathbb{H}$, and, conjecturally, for all $d$ when $K=\mathbb{C}$. This resolves the conjecture of Bukh and Cox \cite{buk18}. Their upper bound for the first moment of an isotropic measure is a direct consequence of this general approach in the case $q=1$.

\item Using upper bounds for isotropic measures we find new lower bounds for a class of the $p$-frame energy potentials over projective codes in real, complex, and quaterionic projective spaces (Theorem \ref{thm:p-energy}).

\item We construct several infinite families of projective codes for which the lower bound on the $p$-frame energy from the previous item is precise (Theorem \ref{thm:constr}). This construction is essentially based on the Gale transform for representatives of projective codes and generalizes Naimark complement frames and the construction of Bukh and Cox.

\end{enumerate}

The paper is structured as follows. In Section \ref{sect:codes}, we provide the general setup for projective codes, briefly describe the linear programming approach and its applications, the Welch bound and the Yudin bound. In Section \ref{sect:tight}, we define tight frames and isotropic measures and discuss their properties. In Section \ref{sect:moments}, we develop a new linear programming approach for bounding energies (particularly, moments) of tight frames and isotropic measures, find new upper bounds of the $q$-th moments of isotropic measures and tight frames, compare the upper bounds, talk about properties of potential minimizers, and discuss computational results. Section \ref{sect:dual} is devoted to the approach of Bukh and Cox connecting minimization problems for projective codes to maximization problems for the dual objects, i.e. tight frames and isotropic measures. There we derive a new lower bound for the $p$-frame energy of projective codes. In Section \ref{sect:constr}, we describe the general construction of projective codes for which the developed lower bound is precise. Section \ref{sect:disc} is devoted to discussing open questions and possible directions for further research in this area. Appendices \ref{app:comp} and \ref{app:orth} contain the proofs of two results from Section \ref{sect:moments}.

\section{Codes in projective spaces}\label{sect:codes}

For the setup on codes in projective spaces we generally follow \cite{coh16}. For $K=\mathbb{R}, \mathbb{C}, \mathbb{H}$ by $K \mathbb{P}^{d-1}$ we mean the set of lines in $K^d$, i.e. $K \mathbb{P}^{d-1}=(K^d\setminus\{0\})/K^{\times}$ defining the equivalence relation by $x\sim x e$ for all $x\in K^d\setminus\{0\}$ and any $e\in K^{\times}$. Note that the right multiplication in this definition is important when $K=\mathbb{H}$ since its multiplication is not commutative.

Each space $K^d$ is equipped with the standard Hermitian product $\langle x, y \rangle = x^* y$, where by $x^*$ we mean the conjugate transpose of $x$. For any element of $K\mathbb{P}^{d-1}$ we will consider its unit-length representative from $K^d$ sometimes abusing the notation. A metric in $K\mathbb{P}^{d-1}$ can be defined by $\rho(x_1,x_2)=\sqrt{1-|\langle x,y \rangle|^2}$, where $x$ and $y$ are such representatives. This metric is topologically equivalent to the normalized Fubini-Study metric: $\vartheta(x_1,x_2)=2 \arccos (|\langle x,y \rangle|)$.

Each element of $x\in K^d$ is associated with the Hermitian matrix $\Pi (x) = x x^*$. For any unit-length vector $x$, $Tr\, \Pi(x) = 1$ and $\Pi^2=\Pi$. The space $\mathcal{H} (K^d)$ of all Hermitian matrices is a real vector space of dimension $d+ \frac {d^2-d} 2 dim_{\mathbb{R}} K$ ($d$ real diagonal elements and $\frac {d^2-d} 2$ elements from $K$ above the diagonal). This space is equipped with the inner product $\langle A, B \rangle = Re\, Tr\,(AB)$. It is fairly easy to check that, for $x,y\in K\mathbb{P}^{d-1}$, $\langle \Pi(x), \Pi(y) \rangle = |\langle x, y \rangle|^2$.

By a \textit{regular simplex} in $K\mathbb{P}^{d-1}$ we mean a set of distinct points $\{x_1,\ldots, x_N\}$ where $|\langle x_i, x_j \rangle|$ are the same for all $i\neq j$. 

\begin{lem}\label{lem:max-simplex}
For a regular simplex $\{x_1,\ldots, x_N\}$ in $K\mathbb{P}^{d-1}$,
$$N\leq d+ \frac {d^2-d} 2 dim_{\mathbb{R}} K.$$

If $N=d+ \frac {d^2-d} 2 dim_{\mathbb{R}} K$,

$$\sum\limits_{i=1}^N \Pi(x_i) = \frac N d I_d.$$
\end{lem}

\begin{proof}
Assume $|\langle x_i, x_j \rangle|=\alpha$ for all $i\neq j$. The Gram matrix formed by $\Pi(x_i)$, $1\leq i \leq N$, is $(1-\alpha^2)I_N + \alpha^2 J_N$, where $I_N$ is the identity matrix and $J_N$ is the matrix of all ones. Since $\alpha^2\in[0,1)$, this matrix is positive definite with rank $N$ which cannot be greater than the rank of $\mathcal{H} (K^d)$, that is $d+ \frac {d^2-d} 2 dim_{\mathbb{R}} K$.

In the case $N=d+ \frac {d^2-d} 2 dim_{\mathbb{R}} K$, matrices $\Pi(x_i)$ form a basis in $\mathcal{H} (K^d)$. Hence $I_d$ can be expressed as their linear combination with real coefficients. Considering $\langle I_d, \Pi(x_i) \rangle$ for all $i$ we can see that these values must be all equal. Comparing the traces we conclude that the coefficients are all equal to $\frac d N$.
\end{proof}

A regular simplex with $d+ \frac {d^2-d} 2 dim_{\mathbb{R}} K$ points in $K\mathbb{P}^{d-1}$ is called a \textit{maximal simplex}.

Linear programming bounds are one of the main instruments for analyzing codes in various spaces. Here we give a short overview of these bounds.

For a metric space $M=K\mathbb{P}^{d-1}$ with its isometry group $G$, as a consequence of the Peter-Weyl theorem, $L^2(M)$ can be decomposed into mutually orthogonal subspaces $V^{(k)}$, where each space defines an irreducible unitary representation of $G$. For each irreducible representation $V^{(k)}$, $dim V^{(k)}=h_k$, one can take an orthonormal basis $\{e_1^{(k)},\ldots,e_{h_k}^{(k)}\}$ and define $P^{(k)}$:

$$P^{(k)} (x,y) = \sum\limits_{i=1}^{h_k} \overline{e_i^{(k)} (x)} e_i^{(k)} (y).$$

Since the representations are unitary, these functions depend only on the distance between its arguments. These functions are called \textit{zonal spherical functions}. The main feature of these functions is that they define positive-definite kernels \cite{sch42}, i.e. for any set of points $x_1,\ldots, x_N$ from M, the matrix $(P^{(k)}(x_i,x_j))$ is positive semidefinite. This immediately follows from the definition of $P^{(k)}$:

\begin{equation}\label{eqn:zsf}
\sum\limits_{i=1}^N \sum_{j=1}^N \overline{a_i} P^{(k)}(x_i,x_j)a_j = \sum\limits_{i=1}^{h_k} \overline{\left(\sum_{j=1}^N e_i^{(k)}(x_j) a_j\right)} \left(\sum_{j=1}^N e_i^{(k)}(x_j)a_j\right) \geq 0.
\end{equation}

A simple consequence of the positive definite condition (\ref{eqn:zsf}) is that the sum of elements of the matrix $\left(P_k(x_i,x_j)\right)$ is non-negative, and this observation is central to the Delsarte method. This method allows one to find kissing numbers in dimensions $8$ and $24$ \cite{lev79,odl79}, the best known asymptotic bounds for kissing numbers and for the density of sphere packings in Euclidean spaces \cite{kab78} (slightly improved in \cite{coh14}), and the general bound for maximal spherical $\varphi$-codes \cite{lev79,lev92} (see also \cite{boy18} for the unified treatment of this approach in various spaces). A certain strengthening of these linear conditions gives new solutions for the kissing problem in $\mathbb{R}^3$ \cite{ans04,mus06}, solution of the problem in $\mathbb{R}^4$ \cite{mus08a}, and the best current bounds for some sphere packing densities \cite{coh03}. The Delsarte method also accounts for the best known asymptotic bounds in some other spaces \cite{mce77,bac06,bar09b}, thereby representing one of the key tools for extremal problems of distance geometry. The Delsarte method has been recently extended to semidefinite programming bounds that rely on a more detailed version of the positivity constraints and on the corresponding positive definite functions on the space \cite{sch05, mus14, mus08b, bac08a, bac09b, bac10}.

It is convenient to express $P^{(k)}$ as a function of $\cos (\vartheta(x,y))$, where $\vartheta(x,y)$ is the normalized Fubini-Study metric. In this case zonal spherical functions are polynomials and $deg\, P^{(k)}=k$. In fact, they appear to be Jacobi polynomials $P_k^{\alpha,\beta}$, where $\alpha = (d-1) \frac {dim_{\mathbb{R}}\,K} {2} - 1$ and $\beta=\frac {dim_{\mathbb{R}}\,K} {2} - 1$.

For our purposes it will be better to express zonal spherical functions as functions of $|\langle x,y \rangle|$ so we define $Q_{K,d}^{(k)} (t) = P^{(k)} (2t^2-1)$, where $P^{(k)}$ is the corresponding polynomial for $K\mathbb{P}^{d-1}$. Then $Q_{K,d}^{(k)}$ are even polynomials of degree $2k$. Sometimes it is convenient to normalize these polynomials so that $Q_{K,d}^{(k)} (1) = 1$ but none of the results will depend on this normalization.

For all choices of $K$, $Q_{K,d}^{(1)}(t) = \frac {dt^2-1} {d-1}$ which immediately implies the uniform packing bound in projective spaces (known as the Welch bound \cite{wel74}, see also \cite{lem73}).

\begin{lem}\label{lem:tight-simplex}
If $x_1,\ldots,x_N\in K\mathbb{P}^{d-1}$ and $\max\limits_{i\neq j} |\langle x_i,x_j\rangle|=\alpha$, then

$$\alpha^2\geq \frac {N-d} {d(N-1)}.$$
\end{lem}

\begin{proof}
Matrix $(Q_{K,d}^{(1)}(|\langle x_i,x_j\rangle|))$ is positive semidefinite so its sum of elements is non-negative:

$$0\leq N+N(N-1)Q_{K,d}^{(1)}(\alpha) = N \left(1+(N-1)\frac {d\alpha^2-1} {d-1}\right).$$

Therefore, $\alpha^2\geq \frac {N-d} {d(N-1)}.$
\end{proof}

Regular simplices reaching this bound are called \textit{tight simplices}.

As mentioned in the introduction, the linear programming bound for the smallest energy is due to Yudin. For the sake of completeness, we include a brief explanation of the Yudin linear programming bound in the case of codes in $K\mathbb{P}^{d-1}$.

\begin{theorem}[Yudin]\label{thm:yudin}
Assume $f(t)\geq h(t) = \sum\limits_{k=0}^l a_{k} Q_{K,d}^{(k)} (t)$ for all $t\in[0,1]$, where $a_k\geq 0$ for all $k$. Then
$$\sum\limits_{i\neq j} f(|\langle y_i,y_j\rangle|) \geq N^2 a_0 - N h (1),$$ where $\{y_1,\ldots,y_N\}$ is a set of unit vectors in $K^d$.
\end{theorem}

\begin{proof}
$$\sum\limits_{i\neq j} f(|\langle y_i,y_j\rangle|) \geq \sum\limits_{i\neq j} h(|\langle y_i,y_j\rangle|) = $$

$$= \sum\limits_{i,j=1}^N h(|\langle y_i,y_j\rangle|) - N h(1) =  \sum\limits_{i,j=1}^N \sum\limits_{k=0}^l a_{k} Q_{K,d}^{(k)}(|\langle y_i,y_j\rangle|) - N h(1) = $$

$$= \sum\limits_{k=0}^l a_{k} \sum\limits_{i,j=1}^N Q_{K,d}^{(k)}(|\langle y_i,y_j\rangle|) - N h(1) =$$

$$= a_0 \sum\limits_{i,j=1}^N 1 + \sum\limits_{k=1}^l a_{k} \sum\limits_{i,j=1}^N Q_{K,d}^{(k)}(|\langle y_i,y_j\rangle|) - N h(1) =$$

$$=  \sum\limits_{k=1}^l a_{k} \sum\limits_{i,j=1}^N Q_{K,d}^{(k)}(|\langle y_i,y_j\rangle|) + N^2 a_0 - N h(1) \geq$$

$$\geq N^2 a_0 - N h(1)$$
because $\sum\limits_{i,j=1}^N Q_{K,d}^{(k)}(|\langle y_i,y_j\rangle|)$ are the sums of values of positive semidefinite matrices for all $k$ and thus are non-negative.
\end{proof}

Yudin \cite{yud92} provided a more general version of this theorem in his paper. He allowed points to have weights (since he used the electrocstatic formulation of the problem, he called them "charges") and proved the weighted form of this result. The version of Yudin may be also interpreted as an energetic lower bound for discrete probability distributions.

\section{Tight frames and isotropic measures}\label{sect:tight}

By a \textit{tight frame} in $K^d$ we mean a finite multiset of vectors $v_1,\ldots, v_N$ satisfying the generalized version of Parseval's identity:

$$\sum\limits_{i=1}^N |\langle x,v_i \rangle|^2=A |x|^2,$$
for any $x\in K^d$. The constant $A$ is called a \textit{frame constant}.

The definition of a tight frame can be rewritten using matrices $\Pi(v)=vv^*$:

$$\sum\limits_{i=1}^N \Pi(v_i) = A I_d.$$

If $D\in K^{N\times N}$ is a matrix formed by the vectors of a tight frame $\{v_1,\ldots, v_N\}$ as columns, then $D D^* = \sum\limits_{i=1}^N v_i v_i^* = A I_d$. Using this we can find the spectral structure of the Gram matrix $D^* D$ of the frame:

\begin{equation}\label{eqn:spec}
D^* D D^* D = D^* (A I_d) D = A D^* D.
\end{equation}

Therefore, the Gram matrix may have only 0 and $A$ as its eigenvalues. The multiplicities of 0 is the dimension of the frame $d$ so the multiplicity of $A$ is $N-d$.

For any $v$, $tr\, \Pi(v) = |v|^2$ so $\sum\limits_{i=1}^N |v_i|^2 = A d$. It is often convenient to consider frames whose average square of norms is unit so that $A=\frac N d$. The example of such frames are \textit{unit norm tight frames} who contain only vectors of unit norm.

If $v_1,\ldots, v_n$ form a unit norm time frame and $|\langle v_i, v_j \rangle| = \alpha$ for all $i\neq j$ then such a frame is called an \textit{equiangular tight frame} (ETF). There is extensive literature devoted to constructions and properties of ETFs (see, for instance, \cite{str03, hol04, sus07} and \cite{bar15} for a more general setup in the real case).

The following lemma explains why equiangular tight frames and tight simplices defined above are essentially the same objects.

\begin{lem}\label{lem:tight-etf}
A set of points in $K\mathbb{P}^{d-1}$ forms a tight simplex if and only if their unit-length representatives form an ETF.
\end{lem}

\begin{proof}
If $\{v_1,\ldots, v_N\}$ is an ETF with $|\langle x_i,x_j \rangle|=\alpha$ for all $i\neq j$ then

$$\frac N d = \sum\limits_{i=1}^N |\langle v_1,v_i \rangle|^2 = 1 + (N-1) \alpha^2$$
so $\alpha^2 = \frac {N-d} {d(N-1)}$ and this set is a tight simplex.

Now assume $\{v_1,\ldots, v_N\}$ is a tight simplex then the proof of Lemma \ref{lem:tight-simplex} implies that\\ $\sum\limits_{i,j=1}^N Q_{K,d}^{(1)} (\langle v_i,v_j\rangle)=0$. From (\ref{eqn:zsf}) we deduce that $\sum_{j=1}^N e_i^{(1)}(v_j) = 0$ for all elements $e_i^{(1)}$ of the orthonormal basis of the corresponding representation $V^{(1)}$. Hence, for any unit-length $x$,

$$\sum\limits_{i=1}^N Q_{K,d}^{(1)} (\langle v_i,x\rangle) = \sum\limits_{i=1}^{h_k} \left(\overline{\sum\limits_{j=1}^N e_i^{(1)}(v_j)}\right) e_i^{(1)}(x) = 0.$$

This implies $\sum\limits_{i=1}^N |\langle v_i,x \rangle|^2=\frac N d |x|^2$ for a unit-length $x$ and, subsequently, for any $x\in K^d$. Therefore, the set $\{v_1,\ldots, v_N\}$ is a tight frame by definition.
\end{proof}

By Lemmas \ref{lem:max-simplex} and \ref{lem:tight-etf}, any maximal simplex is also a tight simplex so we can find the value of its scalar product by Lemma \ref{lem:tight-simplex}:

$$\beta^2=\frac {N-d} {d(N-1)} = \frac {\frac {d^2-d} 2 dim_{\mathbb{R}} K} {d(d+\frac {d^2-d} 2 dim_{\mathbb{R}} K-1)} = \frac 1 {d+2(dim_{\mathbb{R}} K)^{-1}}.$$

For $K=\mathbb{R}$, maximal simplices are known to exist for $d=1, 2, 3, 7, 23$. The necessary condition for their existence in all unknown cases is $d=(2m+1)^2-2$, where $m\in\mathbb{N}$. There are, however, both combinatorial \cite{mak02} and number-theoretic results \cite{ban04,neb12} excluding certain values of $d$ from the list of potential dimensions where maximal simplices exist. The smallest unknown dimension at the moment is $119$. The general bound on the size of a regular simplex in $\mathbb{RP}^{d-1}$ set by Lemma \ref{lem:max-simplex} was improved in \cite{gla18}.

For $K=\mathbb{C}$, maximal simplces are called “symmetric, informationally complete, positive operator-valued measures” (SIC-POVMs). Zauner conjectured that SIC-POVMs exist for all dimensions $d$ \cite{zau99}. This conjecture is known to be true for all $d\leq 21$ and is confirmed numerically for all dimensions up to 181 \cite{fuc17,app19}.

The only known maximal simplex in $\mathbb{HP}^{d-1}$ is the simplex with 15 points in $\mathbb{HP}^2$ constructed by Cohn, Kumar, and Minton in \cite{coh16}. They also conjectured that no other maximal quaternionic simplices exist.

By an \textit{isotropic measure} on $K^d$ we mean a Borel probability measure satisfying the condition $\mathbb{E}_{x\sim\mu} xx^*=\frac 1 d I_{d}$. Note that any uniform distribution over a finite tight frame, with $N$ vectors in $K^d$ and with the frame constant $\frac N d$, is isotropic. Essentially, this notion is a measure-theoretic generalization of tight frames and, as such, it is sometimes known in the literature as a \textit{probabilistic tight frame} \cite{ehl12}.

\section{Moments of isotropic measures}\label{sect:moments}

In this section, we develop a linear programming approach for finding optimizing the $q$-th moments of isotropic measures. Then we use this approach for finding new upper bounds for moments, analyze these bounds, discuss computational results, and talk about potential maximizers reaching the linear programming bound.

For a Borel measure $\mu$ in $K^d$, by its $q$-th moment we mean

$$\mathbb{E}_{x,y\sim\mu} |\langle x, y \rangle|^q =  \int\int |\langle x,y \rangle|^q d\mu(x)d\mu(y).$$

For a tight frame $\{y_1,\ldots,y_N\}\subset K^d$ with the frame constant $\frac N d$, the $q$-energy of the frame is

$$\sum\limits_{i\neq j} |\langle y_i, y_j\rangle|^q.$$

Note that the definitions are slightly different and the $q$-th moment of the uniform distribution over the tight frame is not the same as the $q$-energy of the frame as it also includes the sum of diagonal values, i.e. $\sum\limits_{i=1}^N |\langle y_i, y_i \rangle|^q$.

We are interested in the two types of problems. The first type is measure-theoretic: given the space $K^d$, maximize the $q$-th moment of an isotropic measure in the space. The second type is the discrete problem for tight frames: given the space $K^d$ and the number of vectors $N$ in a tight frame, maximize the $q$-energy of a tight frame.

These problems are motivated by the problems of finding lower packing and energy bounds for projective codes. This connection was discovered by Bukh and Cox who used it for finding new packing bounds. We talk about it in more detail in Section \ref{sect:dual}. Recently similar optimization problems for tight frames and for various energy potentials were also studied in \cite{cah19}.
 
\subsection{Yudin-type linear programming bounds for isotropic measures}

For the following theorem, we use a slightly more general case of energy potentials than the case of the $q$-th moments.

\begin{theorem}\label{thm:moments}
Assume $0\leq f(t)\leq a_0 + a_1 t^2 + \sum\limits_{k\geq 2} a_{k} Q_{K,d}^{(k)} (t)$ for all $t\in[0,1]$, where $a_0, a_1\geq 0$ and $a_{k}\leq 0$ for all $k\geq 2$. For $q\in[1,2]$, let $P(x,y)=|x|^q |y|^q f\left(|\langle \frac x {|x|}, \frac y {|y|} \rangle|\right)$ if neither $x$ nor $y$ is 0 and $P(x,y)=0$ otherwise. Then
$$\mathbb{E}_{x,y\sim\mu} P(x,y) \leq \frac {a_1} {d} + a_0,$$ where $\mu$ is any isotropic measure in $K^d$.
\end{theorem}

\begin{proof}
For the sake of simplicity we will prove this theorem in the case of measures uniformly distributed over a finite support. The set of such measures is weak-* dense in the set of all isotropic measures so this will provide us with the general result as well.

Let $supp(\mu)=\{y_1,\ldots,y_N\}$. We can assume that none of these vectors is 0 because otherwise the uniform distribution over the remaining $N-1$ vectors would give a larger expected value.

For brevity, denote $l_i=|y_i|$, $s_i=y_i/|y_i|$ for all $i$. We also use notation $L$ for the row vector $(l^q_1,\ldots,l^q_N)\in\mathbb{R}^N_+$ and $S$ for the real $N\times N$ matrix $\{|\langle s_i,s_j\rangle|\}$. For any real function $g$, the matrix $\{g(|\langle s_i,s_j\rangle|)\}$ will be denoted by $g(S)$. Then

$$\mathbb{E}_{x,y\sim\mu} |x|^q |y|^q f\left(|\langle \frac x {|x|}, \frac y {|y|} \rangle|\right) =\frac 1 {N^2} \sum\limits_{i,j=1}^N l_i^q l_j^q f\left(|\langle s_i, s_j \rangle|\right)= \frac 1 {N^2} L f(S) L^t \leq$$

$$\leq \frac 1 {N^2} \left(a_0 \left(\sum\limits_{i=1}^N l_i^q\right)^2 + a_1 \sum\limits_{i,j=1}^N l_i^q l_j^q |\langle s_i, s_j \rangle|^2 +  \sum\limits_{k\geq 2} a_{k} L Q_{K,d}^{(k)} (S) L^t \right) \leq$$

\begin{equation}\label{ineq:1}
\leq \frac 1 {N^2} \left(a_0 \left(\sum\limits_{i=1}^N l_i^q\right)^2 + a_1 \sum\limits_{i,j=1}^N l_i^q l_j^q |\langle s_i, s_j \rangle|^2 \right),
\end{equation}
because all matrices $Q_{K,d}^{(k)} (S)$ are positive semidefinite and coefficients $a_k$ are non-positive for $k\geq 2$.

From Jensen's inequality,

\begin{equation}\label{ineq:2}
\frac 1 N \sum\limits_{i=1}^N l_i^q \leq\left(\frac 1 N \sum\limits_{i=1}^N l_i^2 \right)^{q/2}=1.
\end{equation}

By the Cauchy-Schwartz inequality,

$$ \sum\limits_{i,j=1}^N l_i^q l_j^q |\langle s_i, s_j \rangle|^2 =  \sum\limits_{i,j=1}^N l_i l^{q-1}_j |\langle s_i, s_j \rangle|  \times l^{q-1}_i l_j |\langle s_i, s_j \rangle| \leq$$

$$\leq \left(\sum\limits_{i,j=1}^N l^2_i l^{2q-2}_j |\langle s_i, s_j \rangle|^2 \right)^{\frac 1 2} \times  \left(\sum\limits_{i,j=1}^N l^{2q-2}_i l^2_j |\langle s_i, s_j \rangle|^2 \right)^{\frac 1 2}.$$
For a fixed $i$,

$$\sum\limits_{j=1}^N l^{2q-2}_i l^2_j |\langle s_i, s_j \rangle|^2 = l^{2q-4}_i \sum\limits_{j=1}^N l^2_i l^2_j |\langle s_i, s_j \rangle|^2 = l^{2q-4}_i \sum\limits_{j=1}^N |\langle y_i, y_j \rangle|^2 = l^{2q-4}_i \frac {N |y_i|^2} d= \frac N d l^{2q-2}_i.$$
By Jensen's inequality,

$$\frac 1 N \sum\limits_{i=1}^N l_i^{2q-2} \leq\left(\frac 1 N \sum\limits_{i=1}^N l_i^2 \right)^{q-1}=1.$$
Therefore, we get

\begin{equation}\label{ineq:3}
\sum\limits_{i,j=1}^N l_i^q l_j^q |\langle s_i, s_j \rangle|^2 \leq \frac N d \sum\limits_{j=1}^N l^{2q-2}_i \leq \frac {N^2} d.
\end{equation}

The statement of the theorem follows from inequalities (\ref{ineq:1})-(\ref{ineq:3}).

\end{proof}

\begin{remark}
If the isotropic measure is known to be confined to a unit sphere, Theorem \ref{thm:moments} covers essentially all two-point potentials preserved by isometries. Moreover, the conditions of the theorem may be weakened as the signs of the coefficients $a_0$ and $a_1$ become irrelevant.
\end{remark}

The direct analogue of the Yudin bound for codes with a fixed number of points is described in the following corollary.

\begin{corollary}\label{cor:yudin}
Assume $0\leq f(t)\leq h(t) = a_0 + a_1 t^2 + \sum\limits_{k\geq 2} a_{k} Q_{K,d}^{(k)} (t)$ for all $t\in[0,1]$, where $a_0, a_1\geq 0$ and $a_{k}\leq 0$ for all $k\geq 2$. For $q\in[1,2]$, let $P(x,y)=|x|^q |y|^q f\left(|\langle \frac x {|x|}, \frac y {|y|} \rangle|\right)$ if neither $x$ nor $y$ is 0 and $P(x,y)=0$ otherwise. Then
$$\sum\limits_{i\neq j} P(y_i,y_j) \leq N^2 \left(\frac {a_1} {d} + a_0\right) - N h (1),$$ where $\{y_1,\ldots,y_N\}$ is a tight frame in $K^d$ with the frame constant $\frac N d$.
\end{corollary}

\begin{proof}
The uniform distribution over a tight frame in $K^d$ with the frame constant $\frac N d$ is an isotropic measure so, following the proof of Theorem \ref{thm:moments},

$$\sum\limits_{i\neq j} P(y_i,y_j) \leq \sum\limits_{i\neq j}^N |y_i|^q |y_j|^q h\left(|\langle \frac {y_i} {|y_i|}, \frac {y_j} {|y_j|} \rangle\right) =$$

$$= \sum\limits_{i,j=1}^N |y_i|^q |y_j|^q h\left(|\langle \frac {y_i} {|y_i|}, \frac {y_j} {|y_j|} \rangle\right) - \sum\limits_{i=1}^N |y_i|^{2q} h(1) \leq$$

$$\leq  N^2 (\frac {a_1} {d} + a_0) - \sum\limits_{i=1}^N |y_i|^{2q} h(1).$$
The statement of the corollary then follows from Jensen's inequality $\sum\limits_{i=1}^N |y_i|^{2q} \geq N.$
\end{proof}

The conditions of Theorem \ref{thm:moments} are satisfied when $f(t)=t^q$ and $P(x,y)=|\langle x, y \rangle|^q$ so it immediately implies the following corollary.

\begin{corollary}\label{cor:moments}
Given $q\in[1,2]$, assume $t^q\leq a_0 + a_1 t^2 + \sum\limits_{k\geq 2} a_{k} Q_{K,d}^{(k)} (t)$ for all $t\in[0,1]$, where $a_0, a_1\geq 0$ and $a_{k}\leq 0$ for all $k\geq 2$. Then
$$\mathbb{E}_{x,y\sim\mu} |\langle x, y \rangle|^q \leq \frac {a_1} {d} + a_0,$$ where $\mu$ is any isotropic measure in $K^d$.
\end{corollary}

Similarly this works for the case with a fixed number of points.

\begin{corollary}\label{cor:yudin-moments}
Given $q\in[1,2]$, assume $t^q\leq h(t) = a_0 + a_1 t^2 + \sum\limits_{k\geq 2} a_{k} Q_{K,d}^{(k)} (t)$ for all $t\in[0,1]$, where $a_0, a_1\geq 0$ and $a_{k}\leq 0$ for all $k\geq 2$. Then
$$\sum\limits_{i\neq j} |\langle y_i, y_j \rangle|^q \leq N^2 \left(\frac {a_1} {d} + a_0\right) - N h(1),$$ where $\{y_1,\ldots,y_N\}$ is a tight frame in $K^d$ with the frame constant $\frac N d$.
\end{corollary}

Generally, the bounds from Theorem \ref{thm:moments} and Corollaries \ref{cor:yudin}-\ref{cor:yudin-moments} are linear programming bounds with the countable number of variables (coefficients $a_0, a_1, \ldots$) and infinitely many constraints (linear inequalities $f(t)\leq a_0 + a_1 t^2 + \sum\limits_{k\geq 2} a_{k} Q_{K,d}^{(k)} (t)$ must hold for each $t\in [0,1]$). The typical way to find reasonable bounds using such statements is to have a suspect for the optimum with few distances and construct an Hermite interpolant to $f(t)$ with nodes defined by these distances \cite{yud92,coh07}.

\subsection{Applications of the LP bounds}

As the first application of this approach, we can show that tight simplices (ETFs) are optimal for the $|\langle x,y \rangle|^q$ potential, $q\in [1,2]$, among tight frames with the same number of points.

\begin{theorem}\label{thm:ETF_opt}
Given $q\in [1,2]$, for a tight frame $\{y_1,\ldots,y_N\}$ in $K^d$ with the frame constant $\frac N d$,
$$\sum\limits_{i\neq j} |\langle y_i, y_j \rangle|^q \leq (N^2-N) \left(\frac {N-d} {d(N-1)}\right)^{\frac q 2}.$$
For $q\in [1,2)$, the inequality is sharp if and only if the frame consists of unit-length representatives of a tight simplex in $K\mathbb{P}^{d-1}$.
\end{theorem}

\begin{proof}
In order to satisfy the Hermite interpolation conditions, we take $a_1=\frac q 2 \left(\frac {N-d} {d(N-1)}\right)^{\frac {q-2} 2}$ and $a_0=\frac {2-q} 2 \left(\frac {N-d} {d(N-1)}\right)^{\frac q 2}$. It is easy to check that $t^q\leq h(t) = a_1 t^2 +a_0$: taking $s=t^2$, we see that our interpolant is the tangent line to the concave graph of $s^{\frac q 2}$ so it's strictly above it. By Corollary \ref{cor:yudin-moments},

$$\sum\limits_{i\neq j} |\langle y_1, y_j \rangle|^q \leq N^2 (\frac {a_1} d + a_0) - N (a_1 + a_0) =$$

$$=\left(\frac {N-d} {d(N-1)}\right)^{\frac {q-2} 2} \left(N^2 \frac q {2d}  + N^2 \frac {2-q} 2 \frac {N-d} {d(N-1)} - N \frac q 2 - N \frac {2-q} 2 \frac {N-d} {d(N-1)} \right)=$$

$$= \left(\frac {N-d} {d(N-1)}\right)^{\frac {q-2} 2} (N^2-N)\left(\frac {N-d} {d(N-1)}\right)=(N^2-N) \left(\frac {N-d} {d(N-1)}\right)^{\frac q 2}.$$

The inequality is sharp only if all vectors in the frame are of unit length and, for any $i\neq j$, $|\langle y_i, y_j \rangle|^q=h(|\langle y_i, y_j \rangle|)$ which is true only for $|\langle y_i, y_j \rangle|=\left(\frac {N-d} {d(N-1)}\right)^{\frac 1 2}$ if $q<2$.
\end{proof}

For the second application of this approach, we show that maximal simplices (maximal ETFs) are in fact optimal for the $q$-th moments, $q\in [1,2]$, among all isotropic measures. We will need a technical lemma about Hermite interpolants for this result. A similar lemma was used by Yudin for his energy bound \cite{yud92}.

\begin{lem}\label{lem:hermite}
If $H(x)$ is the Hermite interpolant of $f(x)=x^r$, $r\in(0,1]$ satisfying conditions $H(x_i)-f(x_i)=H'(x_i)-f'(x_i)=H(1)-f(1)=0$ for all $x_i$ from the set of given nodes $\{x_1,\ldots,x_M\}\subset(0,1)$, then $H(x)\geq f(x)$ for all $x\in[0,1]$.
\end{lem}

\begin{proof}
Take $W(t)=(t-x_1)^2\ldots (t-x_M)^2(t-1)$. For a given $x\in[0,1]$ such that $W(x)\neq 0$ we can define

$$g(t) = f(t) - H(t) - \frac {f(x)-H(x)} {W(x)} W(t).$$

Then $g(x)=g(x_1)=\ldots=g(x_M)=g(1)=0$ so, by Rolle's theorem, there are $M+1$ points distinct from the nodes of interpolation where $g'$ vanishes. Adding to them $x_1, \ldots, x_M$, where $g'$ vanishes too we get that $g'(t)$ has at least $2M+1$ distinct roots on $[0,1]$. Hence, by Rolle's theorem, there is $\xi\in(0,1)$ such that $g^{(2M+1)}(\xi)=0$. The degree of $H$ is $2M$ so $H^{(2M+1)}(t)=0$ for all $t$. The degree of $W$ is $2M+1$ so $H^{(2M+1)}(t)=(2M+1)!$. Therefore,

$$0=g^{(2M+1)}(\xi) = f^{(2M+1)}(\xi) - \frac {f(x)-H(x)} {W(x)} (2M+1)!;$$

$$f(x)-H(x)=\frac 1 {(2M+1)!} W(x)  f^{(2M+1)}(\xi),$$
which is never positive because $W(x)\leq 0$ and all odd derivatives of $f$ are non-negative.
\end{proof}

Alternatively, one can use the remainder formula for Hermite interpolants to explain Lemma \ref{lem:hermite}.

\begin{theorem}\label{thm:q-measures}
Given $q\in[1,2]$, for any isotropic measure $\mu$ in $K^d$,
$$\mathbb{E}_{x,y\sim\mu} |\langle x, y \rangle|^q \leq \beta^q + \frac {1-\beta^q} {M},$$ where $\beta=\sqrt{\frac {1} {d+2(dim_{\mathbb{R}} K)^{-1}}}$, $M=d+ \frac {d^2-d} 2 dim_{\mathbb{R}} K$. For $q\in[1,2)$, the equality is possible only if there exists a maximal simplex of size $M$ and $\mu$ is then the uniform distribution over a maximal simplex.
\end{theorem}

\begin{proof}
We will find the Hermite interpolant of $t^q$ such that $H(\beta)=\beta^q$, $H'(\beta)=q\beta^{q-1}$, $H(1)=1$, and $H$ is a linear combination of $1$, $t^2$, and $Q_{K,d}^{(2)}(t)$. By Lemma \ref{lem:hermite}, $H(t)\geq t^q$ on $[0,1]$ so in order to use Corollary \ref{cor:moments}, we will just need to check that its coefficients have required signs.

It will be convenient for further calculations to choose the following normalization of $Q_{K,d}^{(2)}(t)$:

$$Q_{K,d}^{(2)}(t) = \left(\frac {dim_{\mathbb{R}} K} {2\beta^2} - \frac {dim_{\mathbb{R}} K} 2 -1\right) + \frac {dim_{\mathbb{R}} K} {\beta^2} (t^2-1) + \frac {\frac {dim_{\mathbb{R}} K} {2\beta^2}+1} {1-\beta^2} (t^2-1)^2.$$
Then $Q_{K,d}^{(2)}(t) = \frac {dim_{\mathbb{R}} K} {2\beta^2} - \frac {dim_{\mathbb{R}} K} 2 -1$, $Q_{K,d}^{(2)}(\beta)=-\beta^2$, and $\dfrac {d\, Q_{K,d}^{(2)}(t)} {d\, t} (\beta)=-4\beta$. The coefficients $a_0$, $a_1$, $a_2$ of the interpolant must satisfy the system of equations

\begin{equation}\label{eqn:hermite}
\begin{cases}
a_0 + a_1 + a_2 \left(\frac {dim_{\mathbb{R}} K} {2\beta^2} - \frac {dim_{\mathbb{R}} K} 2 -1\right) = 1\\
a_0 + a_1 \beta^2 + a_2 (-\beta^2) = \beta^q\\
2a_1 \beta + a_2 (-4\beta) = q\beta^{q-1}
\end{cases}
\end{equation}

From (\ref{eqn:hermite}) we find the coefficients:

$$a_2= \dfrac {\frac {1-\beta^q} {1-\beta^2} - \frac {q\beta^{q-2}} {2}} {\frac {dim_{\mathbb{R}} K} {2\beta^2}+1};$$

$$a_1 = \dfrac {2 \frac {1-\beta^q} {1-\beta^2} + (\frac {dim_{\mathbb{R}} K} {2\beta^2}-1)\frac {q\beta^{q-2}} {2}} {\frac {dim_{\mathbb{R}} K} {2\beta^2}+1};$$

$$a_0 = \beta^q - \beta^2 \dfrac {\frac {1-\beta^q} {1-\beta^2} + \frac {dim_{\mathbb{R}} K} {2\beta^2} \frac {q\beta^{q-2}} {2}} {\frac {dim_{\mathbb{R}} K} {2\beta^2}+1}.$$

On order to show $a_2\leq 0$, we need to check that $0\geq \frac {1-\beta^q} {1-\beta^2} - \frac {q\beta^{q-2}} {2} = \frac {2 - (2-q)\beta^q - q\beta^{q-2}} {2(1-\beta^2)}$. For this we introduce the function $g_q(\beta)=(2-q)\beta^q + q\beta^{q-2}$ and we want to prove that $g_q(\beta)\geq 2$ for all $\beta\in(0,1]$. Note that $g_q(1)=2$ so it is sufficient to prove that $g'_q(\beta)\leq 0$ for all $\beta\in(0,1]$. This is true because

$$g'_q(\beta) = (2-q)q(\beta^{q-1}-\beta^{q-3}) = (2-q)q \beta^{q-3}(\beta^2-1)\leq 0.$$

For the next coefficient we get $a_1\geq 0$ because $\frac {dim_{\mathbb{R}} K} {2\beta^2}-1 = d \frac {dim_{\mathbb{R}} K} {2}.$

Finally, for $a_0\geq 0$ we need to show that

$$\beta^{q-2} \left(\frac {dim_{\mathbb{R}} K} {2\beta^2}+1\right) \geq \frac {1-\beta^q} {1-\beta^2} + \frac {dim_{\mathbb{R}} K} {2\beta^2} \frac {q\beta^{q-2}} {2},$$
which is equivalent to

$$(1-\beta^2)\beta^{q-2} \left(\frac {dim_{\mathbb{R}} K} {2\beta^2} (1-\frac q 2) +1\right)\geq 1-\beta^q.$$
This is true because $\frac {dim_{\mathbb{R}} K} {2\beta^2} (1-\frac q 2)$ is positive and $(1-\beta^2)\beta^{q-2}\geq 1-\beta^q$.

Now we can apply Corollary \ref{cor:moments}:

$$\mathbb{E}_{x,y\sim\mu} |\langle x, y \rangle|^q \leq a_0 + \frac {a_1} d =$$

$$= \beta^q - \dfrac {\beta^2\frac {1-\beta^q} {1-\beta^2} + \frac {dim_{\mathbb{R}} K} {2} \frac {q\beta^{q-2}} {2}} {\frac {dim_{\mathbb{R}} K} {2\beta^2}+1} + \frac 1 d \dfrac {2 \frac {1-\beta^q} {1-\beta^2} + d \frac {dim_{\mathbb{R}} K} {2}\frac {q\beta^{q-2}} {2}} {\frac {dim_{\mathbb{R}} K} {2\beta^2}+1} =$$

$$= \beta^q + \dfrac {(\frac 2 d - \beta^2) \frac {1-\beta^q} {1-\beta^2}} {d \frac {dim_{\mathbb{R}} K} {2}+2} = \beta^q + (1-\beta^q)\dfrac {(\frac 2 d - \frac 1 {d+ 2(dim_{\mathbb{R}} K)^{-1}})} {(d \frac {dim_{\mathbb{R}} K} {2}+2)(1-\frac 1 {d+ 2(dim_{\mathbb{R}} K)^{-1}})}=$$

$$=\beta^q + (1-\beta^q)\dfrac {1} {\frac {d^2 - d} 2 dim_{\mathbb{R}} K +d} = \beta^q + (1-\beta^q) \frac 1 M.$$

Theorem \ref{thm:moments} implies the equality is possible only if the support of the measure belongs to a unit sphere. For $q\in [1,2)$, the only common points between $t^q$ and its interpolant are $\beta$ and 1 so, in order to have the equality, the measure must be a distribution over the simplex with scalar products $\beta e$, where $e\in K^{\times}$. Then the measure is isotropic only if it is uniformly distributed over the vertices of this simplex so this simplex must be tight. From tightness it immediately follows that this is a maximal simplex in $K\mathbb{P}^{d-1}$.

\end{proof}

This theorem gives the proof to Conjecture 29 from \cite{buk18} and its complex and quaternionic analogues.

Although we do not cover the case of octonionic codes in this paper, all the machinery may be used for them as well. For instance, tight simplices from the Cayley plane maximize the $q$-energy for tight frames when $q\in[1,2)$, just like in Theorem \ref{thm:ETF_opt}, and the uniform distribution over the maximal simplex in $\mathbb{OP}^2$ constructed in \cite{coh16} is the unique isotropic measure maximizing the $q$-th moment for $q\in [1,2)$, just like in Theorem \ref{thm:q-measures}.

Using the polynomial from Theorem \ref{thm:q-measures} in Corollary \ref{cor:yudin-moments} immediately gives an upper bound for the $q$-energy in the case of a fixed number of points.

\begin{corollary}\label{cor:q-measures}
Given $q\in [1,2]$, for a tight frame $\{y_1,\ldots,y_N\}$ in $K^d$ with the frame constant $\frac N d$,
$$\sum\limits_{i\neq j} |\langle y_i, y_j \rangle|^q \leq N^2 \left( \beta^q + \frac {1-\beta^q} {M} \right) - N,$$ where $\beta=\sqrt{\frac {1} {d+2(dim_{\mathbb{R}} K)^{-1}}}$, $M=d+ \frac {d^2-d} 2 dim_{\mathbb{R}} K$. For $q\in [1,2)$, the inequality is sharp if and only if the frame consists of unit-length representatives of a maximal simplex in $K\mathbb{P}^{d-1}$.
\end{corollary}

We can compare the bounds from Theorem \ref{thm:yudin} and Corollary \ref{cor:q-measures} to determine which one is better for a particular value of $N$.

\begin{restatable}{lem}{comp}\label{lem:comp}
Given $q\in[1,2)$, for $N<M$,
$$N^2 \left( \beta^q + \frac {1-\beta^q} {M} \right) - N > (N^2-N) \left(\frac {N-d} {d(N-1)}\right)^{\frac q 2};$$
for $N>M$,
$$N^2 \left( \beta^q + \frac {1-\beta^q} {M} \right) - N < (N^2-N) \left(\frac {N-d} {d(N-1)}\right)^{\frac q 2}.$$
The bounds are equal when $q=2$ or when $N=M$ and $q\in[1,2)$.
\end{restatable}

The proof of this lemma is rather technical and, therefore, is moved to Appendix \ref{app:comp}.

Recently Magsino, Mixon, and Parshall \cite{mag19} found another proof of the Bukh-Cox packing bound. They use a linear programming approach which is essentially a special case of the main linear programming bound developed in this section.

A somewhat similar approach for minimizing the energy over projective measures is used in a joint work of the author with Bilyk, Matzke, Park, and Vlasiuk \cite{bil19,bil19a}.

\subsection{Properties of potential optimizers}\label{subsect:pot_opt}

In this subsection, we analyze isotropic measures and tight frames that can be expected to provide exact upper bounds in Theorem \ref{thm:moments} and Corollaries \ref{cor:yudin}-\ref{cor:yudin-moments}. We will consider $f(t)=t^q$, $q\in[1,2)$, though our observations work for a larger class of potentials. We will look at the situation when the auxiliary function $h(t)$, used as an upper bound of the potential function in all these results, has a finite expansion into zonal spherical functions and, generically, all coefficients of the expansion are not zeros:

$$t^q\leq h(t)=a_0 + a_1 t^2 + \sum\limits_{k=2}^T a_{k} Q_{K,d}^{(k)} (t) \text{ for all } t\in[0,1]$$
and $a_0, a_1>0$, $a_i<0$ for all $2\leq i\leq T$.

The number of values $t$ where $t^q=h(t)$ is necessarily finite. In fact, it is definitely no greater than $T+1$ since the $(T+1)$-st derivative of $h(s^{\frac 1 2})-s^{\frac q 2}$ is never 0 (here we take $t^2=s$ and use the fact that $h(t)$ is an even polynomial). We conclude that for measures attaining the linear programming upper bound, there are finitely many possible scalar products on their support, i.e. such measures are necessarily discrete.

Let us analyze the sets of scalar products more accurately. First of all, 0 cannot possibly be a scalar product in an optimal set obtained via the linear programming bounds. Otherwise, since $h(t)$ is an even polynomial, its growth rate is $O(t^2)$ in the neighborhood of 0 so it is definitely smaller than $t^q$. For the bounds over measures, $h(1)$ must be equal to 1 because scalar products of each point with itself have a non-zero contribution into the total $q$-th moment.

Assume the optimizing measure has $m$ distinct scalar products excluding 1 and denote them by $\beta_1$, $\ldots$, $\beta_m$. Then $t^q\leq h(t)$ for all $t\in[0,1]$ and $t^q-h(t)=0$ for $t=\beta_i$, $1\leq i\leq m$, or $t=1$. We can also conclude that $(h(t)-t^q)'=0$ for $t=\beta_i$, $1\leq i\leq m$. Overall, we see that $(h(t)-t^q)'=0$ has at least $2m$ zeros on $[0,1]$: $m$ of them are all $\beta_i$ and at least one from each of the intervals formed by $\beta_i$ and 1. As mentioned above, the $(T+1)$-st derivative has no zeros so $T\geq 2m$.

Since $a_0>0$, the proof of Theorem \ref{thm:moments} implies that all vectors in $supp(\mu)$ must be unit. If we denote these unit vectors by $\{s_1,\ldots,s_N\}$ then, due to our assumption that $a_i<0$ for $2\leq i\leq T$ and following the proof of Theorem \ref{thm:moments}, $\sum\limits_{i,j=1}^N Q_{K,d}^{(k)} (|\langle s_i, s_j \rangle|)=0$ for all $k$, $1\leq k\leq T$ (the condition is not weighted due to \cite{tay95}). Sets satisfying this condition are known in the literature as \textit{T-designs} (see \cite{del77} for the introduction to the spherical case). Generally, designs may be thought of as averaging sets for polynomials over corresponding spaces \cite{sey84}. The case when a $T$-design has $m$ distinct scalar products and $T=2m$, is extremal and sets of this kind are called \textit{tight designs}. The classification of tight designs in projective spaces is not complete (see \cite{lyu09}) but in all known cases, except for maximal simplices and diagonals of regular $(4N+2)$-gons, tight designs have 0 scalar products which, as we established, is impossible in our case.

The situation for maximizing the $q$-energy over tight frames of a fixed size is similar. The only difference between the measure case is the value of $h$ at 1. Now we do not require $h(1)=1$ so the number of possible zeros of $(h(t)-t^q)'$ is less by one. This in turn changes the design strength restriction to $T\geq 2m-1$. Sets of this kind are known as \textit{sharp configurations} and are proven to be universally optimal, i.e. minimizing all absolutely monotonic potentials of scalar products \cite{coh07}. Except for tight simplices and diagonals of regular $(4L+2)$-gons, all known sharp configurations have 0 scalar products (see \cite{coh16}).

The cases of maximal and, more generally, tight simplices were already observed in the paper so here we say a few words about the remaining case of regular polygons. Firstly, we already know that the uniform distribution over three diagonals of a regular hexagon is an optimal measure for the $q$-th moments so, whenever $4L+2$ is divisible by 3, the maximum is achieved on some number of copies of this set of three lines. Computational results show that even for $4L+2$ not divisible by 3, linear programming bounds cannot show the optimality of any other regular polygon as one may expect at least for small values of $L$.

Having all this checked, it is important to understand that the initial assumptions of genericness are quite strict. It may happen that for some specific configuration and some specific $q$ there is a linear programming bound with some of $a_i$ equal to 0. It seems unlikely, however, for a configuration like this to maximize the $q$-energy for all $q\in[1,2]$ as it happens in the case of tight simplices.

Of course here we discuss only the optimality that could be shown by linear programming bounds. Optimal tight frames still can have 0 as their scalar products and many of the usual suspects (e.g. projective sets formed by shortest vectors of $E_8$ or the Leech lattices) may be optimal as tight frames as well. However, we can show that they cannot be optimal as measures and the results similar to Theorem \ref{thm:q-measures} are not possible for them.

\begin{restatable}{theorem}{orth}\label{thm:orth}
For any $q\in[1,2)$, a discrete isotropic measure in $K^d$ with two orthogonal vectors cannot be a local maximum for the $q$-th moment over the set of all isotropic measures. 
\end{restatable}

The idea of the proof is to perturb vectors in an orthogonal pair. For measures, it is possible to implement an $O(\delta)$-perturbation of one of these vectors and $O(\delta^2)$-perturbations of all other vectors while keeping the measure isotropic. If the problem is to optimize for discrete isotropic measures with a fixed number of points, this kind of perturbation is not possible. It is quite probable that tight designs or, more generally, sharp sets are optimizers for this setup as well. We also conjecture that a similar perturbation technique can work for non-discrete optimizers and the support of the optimal measure cannot have two orthogonal vectors in the general case either.

The complete proof of Theorem \ref{thm:orth} is quite technical and thus is shown in Appendix \ref{app:orth}.

\subsection{Computational results}\label{subsect:comp}

In this subsection we discuss computational results that can be achieved by using Corollaries \ref{cor:moments} and \ref{cor:yudin-moments}. When fixing the degree $T$ of a polynomial $h(t)$, the conditions become linear with respect to the coefficients $a_0$, $a_1$, $\ldots$, $a_T$: $a_0\geq 0$, $a_1\geq 0$, $a_2\leq 0$, $\ldots$, $a_T\leq 0$, and $a_0+a_1 t+\ldots + a_T Q_{K,d}^{(k)}(t)-t^q\geq 0$ for all $t\in [0,1]$. The function to be optimized is linear with respect to the coefficients: $a_0+\frac {a_1} d$ in the case of Corollary \ref{cor:moments} and $N^2 \left(a_0+\frac {a_1} d \right) - N(a_0+a_1+\ldots+a_T)$ in the case of Corollary \ref{cor:yudin-moments}. Hence, generally speaking, this is a linear program with infinitely many linear constraints.

The situation changes drastically when the inequalities may be interpreted as polynomial. The non-negativity of a polynomial can be transformed to a sum-of-squares condition. Indeed, by the Markov-Luk\'{a}cs theorem \cite{mar06,luk18}, any polynomial $p(t)$ which is non-negative on $[0,1]$ can be represented as $f^2(t)+(t-t^2)g^2(t)$, where $f(t)$ and $g(t)$ are two real polynomials with degrees no greater than $T'$ and $T'-1$, respectively. Then a new polynomial

$$q(t)=(t^2+1)^{2T'}p\left(\frac {t^2} {t^2+1}\right) = \left((t^2+1)^{T'}f\left(\frac {t^2} {t^2+1}\right)\right)^2+\left(t(t^2+1)^{T'-1}g\left(\frac {t^2} {t^2+1}\right)\right)^2$$
is clearly represented as a sum of two squares. A polynomial $q(t)$ of degree $2M'$ is a sum of squares if and only if there exists a positive semidefinite matrix $Q$ such that $q(t)=XQX^t$, where $X=(1,t,\ldots,t^{M'})$ (see the general result of Nesterov \cite{nes98}) and thus there are positive semidefinite constraints on the coefficients of $h(t)$. 

We used this approach for the first moment of isotropic measures and tight frames. As an outcome of the observations above, a polynomial inequality $h(t)-t\geq 0$ can be transformed into an semidefinite programming (SDP) problem on coefficients of the polynomial. This SDP problem can be solved by computer. We used the SOSTOOLS toolbox for Matlab \cite{pra02} with the SeDuMi SDP-solver \cite{stu99} to get computational results. For Corollary \ref{cor:moments}, we used polynomials of degree 14 and checked $\mathbb{R}^d$ for $d\leq 25$, $\mathbb{C}^d$ for $d\leq 20$, and $\mathbb{H}^d$ for $d\leq 15$. All upper bounds for isotropic measures obtained this way coincided with the bound from Theorem \ref{thm:q-measures}. The optimizing polynomials $h(t)$ were precisely those from the proof of the theorem with one notable exception. The maximal simplex in $\mathbb{CP}^1$ has an absolute value of the scalar product of $\frac 1 {\sqrt{3}}$. If $\{s_1, s_2, s_3, s_4\}\subset\mathbb{C}^2$ are unit representatives of its elements, it is easy to check that $\sum\limits_{i,j=1}^4 Q_{\mathbb{C},2}^{(5)} (|\langle s_i, s_j \rangle|)=0$ (it is a so-called projective $\{5,2,1\}$-design). This means the polynomial $Q_{\mathbb{C},2}^{(5)}$ of degree 10 may also appear in the interpolation and this is exactly what was observed when solving the program computationally.

Using a very similar approach based on Corollary \ref{cor:yudin-moments}, we solved the problem of maximizing the 1-energy over tight frames with a fixed number of points for small values of parameters. We used polynomials of degree 14 and checked $\mathbb{R}^d$ for $d\leq 15$, $\mathbb{C}^d$ for $d\leq 10$, and $\mathbb{H}^d$ for $d\leq 6$. For each of these spaces, we computed the upper bounds for the 1-energy for $N$ points where $N$ varied from $d+1$ to $d+\frac {d^2-d} 2 dim_{\mathbb{R}} K + 5$. In all of these cases, except for the case of 5 points in $\mathbb{R}^3$, the upper bound was the lower of the bounds given by Theorem \ref{thm:yudin} and Corollary \ref{cor:q-measures}, thereby confirming Lemma \ref{lem:comp}. In all the observed cases, apart from the codes in $\mathbb{C}^2$ and the exception mentioned above, optimizing polynomials were precisely those used in Theorem \ref{thm:yudin} and Theorem \ref{thm:q-measures}: quadratic Hermite interpolants from Theorem \ref{thm:yudin} when $N<d+\frac {d^2-d} 2 dim_{\mathbb{R}} K$, Hermite interpolants of degree 4 from Theorem \ref{thm:q-measures} when $N>d+\frac {d^2-d} 2 dim_{\mathbb{R}} K$, and a non-negative linear combination of these two when $N=d+\frac {d^2-d} 2 dim_{\mathbb{R}} K$ as both polynomials provide the same result under this condition. Just as for the measure case, optimizing polynomials for tight frames in $\mathbb{C}^2$ had non-zero coefficients for $Q_{\mathbb{C},2}^{(5)}$.

The exceptional case is for tight frames of 5 points in $\mathbb{R}^3$. The optimizing polynomial obtained computationally is $0.148245 + 1.377915 t^2 - 0.117231 Q_{\mathbb{R},3}^{(4)}(t)$. The upper bound for the 1-energy found via this polynomial is approximately $8.144098$ which is better than the bound one can find using Theorem \ref{thm:yudin}, $\frac {20} {\sqrt{6}}\approx 8.164966$. We have no explanation for this exception and do not know how close the actual maximal 1-energy is to this numerical bound.

Finding bounds for values of $q\in(1,2)$ is more complicated because we cannot assume that $h(t)-t^q$ is a polynomial. This can be overcome for rational $q=\frac m n$ by considering $h(t^n)-t^m$. If $m$ and $n$ are large, the corresponding SDP problem becomes more complicated to solve. For our computations, we considered $q=\frac 3 2$ and polynomials of degree 10. Even for polynomials of degree 14 or 12, the SDP solver does not provide results converging to reasonable bounds. With this setup, we found upper bounds for the $\frac 3 2$-moment of isotropic measures in the same spaces as for the first moment: $\mathbb{R}^d$ for $d\leq 15$, $\mathbb{C}^d$ for $d\leq 10$, and $\mathbb{H}^d$ for $d\leq 6$. Just like in the first moment case, computational results agree with the bounds from Theorem \ref{thm:q-measures}. Optimizing polynomials coincide with those used in the proof of Theorem \ref{thm:q-measures} with the one exception of $\mathbb{C}^2$, where polynomials of degree 10 show up as well.

We also calculated upper bounds for the $\frac 3 2$-energy among tight frames of a given size. We used the same spaces and sizes of tight frames as in the computations for the $1$-energy. The results are, generally, very similar. In each case, except for 5 points in $\mathbb{R}^3$, the bound agrees with the bounds given by Theorem \ref{thm:yudin} and Corollary \ref{cor:q-measures} simultaneously confirming Lemma \ref{lem:comp}. Optimizing polynomials are again, with the exception of $\mathbb{C}^2$ and 5 points in $\mathbb{R}^3$, those used in the proofs of  Theorem \ref{thm:yudin} and Corollary \ref{cor:q-measures} and their linear combinations for the borderline case $N=d+\frac {d^2-d} 2 dim_{\mathbb{R}} K$. In $\mathbb{C}^2$, polynomials of degree 10 show up. The minimizing polynomial for 5 points in $\mathbb{R}^3$ is of degree 8 but it is different from the one for the 1-energy so we expect the maximizing tight frame to be different for different values of $q$.

\subsection{Upper bound for the $\infty$-moment}

All the exact upper bounds found so far are attained exclusively on tight frames with unit vectors. This is not generally the case. In the last part of the section, we show an example of the exact upper bound where some of the vectors of the optimizing set are not always unit. Here we prove an upper bound for $\max|\langle x,y \rangle|$, i.e. the $\infty$-moment, of an isotropic measure with a finite support ($d$-dimensional tight frame with $N$ points and the frame constant $\frac N d$).

\begin{theorem}\label{thm:inf}
For a tight frame $\{y_1,\ldots,y_N\}$ in $K^d$ with the frame constant $\frac N d$,

$$\max\limits_{i\neq j}|\langle y_i,y_j \rangle|\leq \frac N {2d}.$$
\end{theorem}

\begin{proof}
We use the tight frame condition for $y_i$:

$$\frac N d |y_i|^2 = |y_i|^4 + \sum\limits_{j\neq i}|\langle y_i, y_j\rangle|^2.$$

From this we conclude that for any $j$, $j\neq i$,

$$|\langle y_i, y_j\rangle|^2\leq \frac N d |y_i|^2 - |y_i|^4 =\frac {N^2} {4d^2} - \left(\frac N {2d} - |y_i|^2\right)^2 \leq \frac {N^2} {4d^2}.$$

Hence $|\langle y_i, y_j\rangle| \leq \frac N {2d}$ for any pair $i\neq j$.
\end{proof}

The equality in Theorem \ref{thm:inf} is possible only if $|y_i|^2=|y_j|^2=|\langle y_i, y_j\rangle|=\frac N {2d}$, i.e. $y_i$ and $y_j$ are representatives of the same point in $K\mathbb{P}^{d-1}$ and, unless $N=2d$, are of the same non-unit length. The remaining vectors must be orthogonal to $y_i$ and $y_j$ and form a tight frame with the same frame constant $\frac N d$ in the orthogonal subspace. It is also clear that any tight frame with this structure will be a maximizer.

\section{Optimizing in the dual space}\label{sect:dual}

In this section we extend the approach of Bukh and Cox and prove energy bounds for projective codes using the upper bounds for moments of isotropic measures. Generally, all these bounds will work for a larger set of problems. In particular, we consider square matrices from $K^{N\times N}$ with all diagonal elements equal to 1 and of rank $\leq d$. This set contains the set of Gram matrices defined by $N$ unit vectors in $K^d$, although we do not require matrices to be Hermitian and positive semidefinite which would be necessary for Gram matrices.

In this section we obtain new lower bounds for $\sum\limits_{i\neq j} |A_{ij}|^p$, where $A$ is an arbitrary matrix under the constraints described above. In case $A$ is a Gram matrix of a set of unit vectors in $K^d$, this sum is known as a $p$-frame energy (see \cite{ehl12, gla19, che19}). Since the value of this sum is equal for all unit-length representatives of a projective code in $K\mathbb{P}^{d-1}$, it is natural to consider the minimization problem for sets of vectors as a minimization problem for projective codes. For even natural $p$, the lower bounds for the $p$-frame energy that are precise for large enough projective codes follow from \cite{wel74} in the complex case and from \cite{sid74, ven01} for the real case. The equality is attained on projective designs mentioned in Subsection \ref{subsect:pot_opt}.

\subsection{Duality of matrices of a given rank and tight frames}

The following theorem was essentially proven by Bukh and Cox in their paper (the proof was given for the case $p=\infty$ and the general case was outlined in the discussion section).

\begin{theorem}\label{thm:p-bound}
For any matrix $A\in K^{N\times N}$ of rank $d$ with $A_{ii}=1$ for all $i$, there exists a tight frame $\{y_1,\ldots, y_N\}\subset K^{N-d}$ with the frame constant $\frac N {N-d}$ such that

$$\left(\sum\limits_{i\neq j} |A_{ij}|^p\right)^{\frac 1 p} \left( \sum\limits_{i\neq j} |\langle y_i, y_j \rangle|^q \right)^{\frac 1 q} \geq N$$

for any pair of $p,q\in [1,+\infty]$ such that $\frac 1 p + \frac 1 q =1$.
\end{theorem}

\begin{proof}
The kernel of the matrix $A$ has rank $N-d$ so it has the basis of $N$-dimensional vectors $\xi_1, \ldots, \xi_{N-d}$. All vectors of $Ker A$ can be written as $v_1 \xi_1+\ldots + v_{N-d} \xi_{N-d}$, where $v_i$ are arbitrary numbers from $K$. Using the notation $v=(v_1,\ldots,v_{N-d}) \in K^{N-d}$ we get that $Ker\, A = \{ \langle z_1, v\rangle,\ldots, \langle z_N, v \rangle | v\in K^{N-d}\}$ for fixed vectors $z_1,\ldots,z_N\in K^{N-d}$.

The Hermitian form $|\langle z_1, v\rangle |^2 + \ldots +|\langle z_N, v \rangle|^2$ is positive definite because the vectors $z_i$ span $K^{N-d}$. Hence there exists $M\in GL_{N-d}(K)$ such that this form is $\frac N {N-d} \langle Mv, Mv \rangle$. Then

$$|\langle (M^*)^{-1}z_1, v\rangle |^2 + \ldots |\langle (M^*)^{-1}z_N, v \rangle|^2 = |\langle z_1, M^{-1} v\rangle |^2 + \ldots |\langle z_N, M^{-1} v \rangle|^2 = \frac N {N-d} |v|^2.$$
Vectors $y_i=(M^*)^{-1}z_i$, $1\leq i\leq N$, form a tight frame with the frame constant $\frac N {N-d}$. $Ker\, A$ can be still written as $\{ \langle y_1, v\rangle,\ldots, \langle y_N, v \rangle | v\in K^{N-d}\}$.

For the next step, we take $v=y_i$ and denote $Y_i=(\langle y_1, y_i \rangle,\ldots, \langle y_N, y_i\rangle)$. Then for $A_i$, the $i$-th row of $A$, it must hold $\langle A_i, Y_i \rangle = 0$. Therefore,

$$\langle y_i, y_i \rangle = \left|\sum\limits_{j\neq i} A_{ij} \langle y_j, y_i \rangle\right| \leq \sum\limits_{j\neq i} |A_{ij}| |\langle y_j, y_i \rangle|.$$
Summing up such inequalities for all $i$ we get that

$$\sum\limits_{i=1}^N \langle y_i, y_i \rangle \leq  \sum\limits_{i=1}^N \sum\limits_{j\neq i} |A_{ij}| |\langle y_j, y_i \rangle| \leq \left(\sum\limits_{i\neq j} |A_{ij}|^p\right)^{\frac 1 p} \left( \sum\limits_{i\neq j} |\langle y_i, y_j \rangle|^q \right)^{\frac 1 q}$$
for $p,q\in (1,+\infty)$, $\frac 1 p + \frac 1 q =1$, by the H\"{o}lder inequality. Taking into account that $\sum\limits_{i=1}^N \langle y_i, y_i \rangle = N$ for a tight frame with the frame constant $\frac {N} {N-d}$, we get the statement of the theorem. The cases $p=+\infty, q=1$ and $p=1, q=+\infty$ follow by taking corresponding limits.
\end{proof}

Although initially the set of vectors $\{z_1,\ldots, z_N\}$ in the proof of Theorem \ref{thm:p-bound} was defined up to all non-singular linear transformations, the set $\{y_1,\ldots, y_N\}$ is unique up to isometries due to the tight frame condition and the fixed frame coefficient.

If a matrix $A$ is a Gram matrix of a set of unit vectors $X=\{x_1,\ldots,x_N\}\subset K^d$, one way to understand the construction of the tight frame $Y=\{y_1,\ldots,y_N\}$ in Theorem \ref{thm:p-bound} is to interpret $Y$ as an isotropic representative of the Gale transform of the set $X$. Here we give a brief explanation for this interpretation.

The notion of Gale duality goes back to Gale \cite{gal56} who used this construction for analyzing combinatorial properties of polytopes, although similar approaches were known in mathematics for quite some time \cite{cob15}. Here we define Gale duality following the literature on polytopes. For an ordered set of vectors $X=\{x_1,\ldots,x_N\}\subset K^d$ of full rank, the set of linear dependence relations on $X$ are all vectors $v=(v_1,\ldots,v_N)^*\in K^N$ such that $v_1x_1+\ldots+v_Nx_N = 0$. The set of linear dependence relations forms a subspace of dimension $N-d$ in $K^N$. We can choose an arbitrary basis of this subspace and form a matrix $B\in K^{N\times(N-d)}$ with the basis vectors as columns. The ordered set of $N$ columns of $B^*$ is then said to be \textit{Gale dual} to the set $X$. If $C$ is a matrix formed by vectors from $X$ as columns using their initial order, then the necessary and sufficient condition on $B$ is that $CB=0$ and the rank of $B$ is $N-d$. We note that the condition $CB=0$ can be substituted by $C^* C B=0$ which is essentially what was used in the proof of Theorem \ref{thm:p-bound}. Hence the set constructed there is precisely a Gale dual set of the set $X$, when $A=C^* C$, i.e. $A$ is a Gram matrix of the set of unit vectors. Initially, a Gale dual set is not defined uniquely and we will use the term of Gale duality for any pair of Gale dual sets and Gale transform for any set dual to a given one in this non-unique sense for the rest of the paper. In the proof of the theorem, we impose the isotropic condition which implies the uniqueness of the dual set up to isometries and the uniqueness of its Gram matrix $B B^*$. We should also mention that the (metric) Gale duality in \cite{coh16} is a special case of the Gale duality used in Theorem \ref{thm:p-bound}.

\begin{corollary}\label{cor:p-energy}
For any matrix $A\in K^{N\times N}$ of rank $d$ with $A_{ii}=1$ for all $i$, there exists an isotropic measure $\mu$ in $K^{N-d}$ with $|supp(\mu)|=N$ such that

$$\left(\sum\limits_{i\neq j} A^p_{ij}\right)^{\frac 1 p} \geq \frac {N} {(N^2 \mathbb{E}_{x, y\sim\mu}|\langle x, y \rangle|^q-N)^{\frac 1 q}}.$$
for any pair of $p,q\in [1,+\infty]$ such that $\frac 1 p + \frac 1 q =1$.
\end{corollary}

\begin{proof}
For $\mu$ we can take the uniform distribution over the tight frame $\{y_1,\ldots,y_N\}$ constructed in the proof of Theorem \ref{thm:p-bound}. Since the frame constant is $\frac {N} {N-d}$, this distribution defines an isotropic measure in $K^{N-d}$. By Jensen's inequality, 
$$\frac 1 N \sum\limits_{i=1}^N \langle y_i,y_i \rangle^q \geq \left(\frac 1 N \sum\limits_{i=1}^N \langle y_i,y_i \rangle \right) = 1.$$
Therefore,

$$N^2 \mathbb{E}_{x, y\sim\mu}|\langle x, y \rangle|^q-N \geq \sum\limits_{i\neq j} |\langle y_i, y_j \rangle|^q$$
and the corollary follows from Theorem \ref{thm:p-bound}.
\end{proof}

\subsection{Lower bounds for the $p$-frame energy}

As a first application of Theorem \ref{thm:p-bound}, we give a new proof of the result from \cite{okt07} (also Proposition 3.1 in \cite{ehl12}). A different proof using the approach of Bukh and Cox was given in \cite{gla19}. 

\begin{proposition}\label{prop:p>2}
For any $p\geq 2$ and any matrix $A\in K^{N\times N}$ of rank $d$ with $A_{ii}=1$ for all $i$,
$$\left(\sum\limits_{i\neq j} A^p_{ij}\right)^{\frac 1 p}\geq (N(N-1))^{\frac 1 p} \left( \frac {N-d} {d(N-1)} \right)^{\frac 1 2}.$$
\end{proposition}

\begin{proof}
By Theorem \ref{thm:p-bound} there exists a tight frame $\{y_1,\ldots,y_N\}\subset K^{N-d}$ with the frame constant $\frac N {N-d}$ such that

$$\left(\sum\limits_{i\neq j} |A_{ij}|^p\right)^{\frac 1 p} \left( \sum\limits_{i\neq j} |\langle y_i, y_j \rangle|^q \right)^{\frac 1 q} \geq N$$
By Theorem \ref{thm:ETF_opt},
$$\sum\limits_{i\neq j} |\langle y_1, y_j \rangle|^q \leq (N^2-N) \left(\frac {d} {(N-d)(N-1)}\right)^{\frac q 2}.$$
Using these two inequalities we get the required lower bound.
\end{proof}

Tracking the inequalities that lead to this bound we can see that, given $p>2$, the bound is sharp if an only if $A$ is a Gram matrix of unit representatives of a tight simplex in $K\mathbb{P}^{d-1}$.

In a similar way, combining Corollary \ref{cor:p-energy} with Theorem \ref{thm:q-measures}, we immediately obtain the universal bound for the $p$-energy of a matrix depending on its size and rank.

\begin{theorem}\label{thm:p-energy}
For any matrix $A\in K^{N\times N}$ of rank $d$ with $A_{ii}=1$ for all $i$, for $p\geq 2$,

$$\left(\sum\limits_{i\neq j} |A_{ij}|^p\right)^{\frac 1 p} \geq \frac {N} {\left(N^2 \left(\beta^{\frac p {p-1}} + \frac 1 M (1-\beta^{\frac p {p-1}})\right)-N\right)^{1-\frac 1 p}},$$

where $\beta=\sqrt{\frac {1} {N-d+2(dim_{\mathbb{R}} K)^{-1}}}$, $M=N-d+ \frac {(N-d)^2-(N-d)} 2 dim_{\mathbb{R}} K$.
\end{theorem}

Note that this bound may be sharp only if there exists a maximal simplex in dimension $N-d$ and there is an isotropic measure over $N$ points in $K\mathbb{P}^{N-d-1}$ coinciding with the isotropic measure over such a simplex, i.e. $N$ is divisible by the size of the maximal simplex $M$, where $M=N-d+ \frac {(N-d)^2-(N-d)} 2 dim_{\mathbb{R}} K$. In the next section we will show that these conditions are also sufficient and, if they are satisfied, the bound of the theorem is sharp.

Finally, we prove the result dual to Theorem \ref{thm:inf}.

\begin{theorem}
For any matrix $A\in K^{N\times N}$ of rank $d$ with $A_{ii}=1$ for all $i$,

$$\sum\limits_{i\neq j} |A_{ij}| \geq 2(N-d).$$
The bound is sharp if only if $d\leq N\leq 2d$ and all $N$ indices can be partitioned into $N-d$ pairs and $2d-N$ singletons such that $A_{ij}=0$ when $(i,j)$ is not a pair and $A_{ij}=A_{ji}=\pm 1$ when $(i,j)$ is a pair in this partition.
\end{theorem}

\begin{proof}
The Gale dual tight frame $\{y_1,\ldots,y_N\}\in K^{N-d}$ constructed via Theorem \ref{thm:p-bound} must satisfy Theorem \ref{thm:inf}. Thus $\max\limits_{i\neq j} |\langle y_i,y_j \rangle|\leq \frac N {2(N-d)}$ and, by Theorem \ref{thm:p-bound},

$$\sum\limits_{i\neq j} |A_{ij}| \frac N {2(N-d)}\geq N.$$
Therefore, $\sum\limits_{i\neq j} |A_{ij}| \geq 2(N-d)$.

The bound is sharp when for each $A_{ij}\neq 0$, the corresponding $|\langle y_i,y_j \rangle|$ is precisely $\frac N {2(N-d)}$. Due to the proof of Theorem \ref{thm:inf}, this can happen only if $|y_i|^2=|y_j|^2=\frac N {2(N-d)}$ and all other vectors are orthogonal to both $y_i$ and $y_j$. This means $A_{ik}=A_{jk}=A_{ki}=A_{kj}=0$ for all $k\neq i,j$. $\sum\limits_{i\neq j} |A_{ij}|$ is minimal possible so changing $A_{ij}$ and $A_{ji}$ to 0 must change the rank of the matrix. This is posssible only if both $A_{ij}$ and $A_{ji}$ are the same and are equal to either 1 or -1. The number of pairs cannot be greater than $d$ so $N\leq 2d$. This concludes the proof for the sharp bounds.
\end{proof}

A slightly more general result was proven in \cite{gla19} without the use of moments of isotropic measures.

\section{Constructing sharp codes}\label{sect:constr}

In this section we show how to construct sharp codes for bounds from Theorems \ref{thm:p-bound} and \ref{thm:p-energy}. Our construction builds upon the construction of Bukh and Cox in \cite{buk18} and the construction of Gale dual projective codes \cite{coh16} or Naimark complements of tight frames \cite{cas13}.

Let $C\in K^{n\times n}$ be an Hermitian matrix with unit elements on the diagonal and a maximal eigenvalue equal to $\lambda$ with multiplicity $k$. We assume $C\neq I_n$ so $\lambda$, as the maximal eigenvalue of $C$, must be greater than 1. For a given natural number $b$, we will look for Gram matrices from $K^{bn\times bn}$ of a set of unit vectors in $K^{bn-k}$ of the following type:

$$C(\alpha,\beta,\gamma)=\alpha I_{nb} + \beta I_n\otimes J_b + \gamma C\otimes J_b,$$
where $\otimes$ is a standard tensor product, $J_b$ is a $b\times b$ matrix of all ones, and $\alpha,\beta,\gamma\in\mathbb{R}$.

Knowing the eigenvalues of $C$, it is clear how to find necessary and sufficient conditions on $\alpha$, $\beta$, and $\gamma$. For instance, for $b=1$, such a matrix will be $\frac {\lambda} {\lambda-1} I_n - \frac 1 {\lambda-1} C$.

\begin{lem}\label{lem:family}
Given $b>1$, for any $\alpha\in[0,\frac b {b-1}]$, $C\left(\alpha,\frac {-b\lambda + (b\lambda-1)\alpha} {b(1-\lambda)},\frac {b+(1-b)\alpha} {b(1-\lambda)}\right)$ is a Gram matrix of the set of unit vectors in $K^{bn-k}$.
\end{lem}

\begin{proof}
We use the same notation as above, i.e. $\beta=\frac {-b\lambda + (b\lambda-1)\alpha} {b(1-\lambda)}$ and $\gamma=\frac {b+(1-b)\alpha} {b(1-\lambda)}$.

Let $C$ have eigenvalues $\lambda_1<\ldots<\lambda_l<\lambda$ with multiplicities $k_1, \ldots, k_l, k$, respectively. Then the eigenvalues of $\beta I_n + \gamma C$ are $\beta + \gamma \lambda_1$, $\ldots$, $\beta+\gamma \lambda$ with their respective multiplicities. For any matrix, its tensor product with $J_b$ has eigenvalues obtained by multiplying the eigenvalues of the initial matrix by $b$ and multiplicities of these eigenvalues are the same as those of their counterparts in the initial matrix. The remaining eigenvalue is 0 with the suitable multiplicity. Therefore, the matrix $(\beta I_n + \gamma C)\otimes J_b$ has eigenvalues $(\beta + \gamma \lambda_1)b$, $\ldots$, $(\beta + \gamma \lambda)b$ with multiplicities $k_1,\ldots, k$, respectively, and the eigenvalue 0 with the multiplicity $bn-n$. Finally, the matrix $C(\alpha,\beta,\gamma)$ has eigenvalues $\alpha+(\beta + \gamma \lambda_1)b$, $\ldots$,  $\alpha+(\beta + \gamma \lambda)b$ with their respective multiplicities $k_1,\ldots, k$, and $\alpha$ with the multiplicity $bn-n$.

In order for the matrix to be a Gram matrix of the set of unit vectors from $K^{bn-k}$, it must be Hermitian, with non-negative eigenvalues, with the 0 eigenvalue of multiplicity k, and have the diagonal of all ones. For the eigenvalues, it is sufficient to check that $\alpha\geq 0$, $\alpha+(\beta + \gamma \lambda)b = 0$, $\alpha+(\beta + \gamma \lambda_i)b\geq 0$ for all $i$ from 1 to $l$, $\alpha\geq 0$. The first two conditions are clearly satisfied. Since $\gamma\leq 0$, the last condition is satisfied too because $\lambda_i<\lambda$ for all $i$. The remaining condition is $\alpha+\beta+\gamma=1$ for all diagonal elements. This one is clearly satisfied as well.
\end{proof}

We also note that the constructed set is full-dimensional unless $\gamma=0$, in which case the dimension is $nb-n$ and the set is the union of $n$ pairwise orthogonal $(b-1)$-dimensional regular simplices.

The main idea now is to use a Gram matrix $C$ of a tight frame for the construction in Lemma \ref{lem:family}. If the frame is in $K^{d-k}$, the frame constant is then the maximal eigenvalue of the Gram matrix and its multiplicity is $k$ precisely. $b$ copies of such a frame form a tight frame in $K^{d-k}$ with $bN$ vectors. The Gram matrix of this tight frame is $C\otimes J_b$. The whole one-parametric family defined in Lemma \ref{lem:family} belongs to the general Gale transform of the frame. It is easy to confirm by checking that $C(\alpha,\beta,\gamma) (C\otimes J_b) = 0$. Indeed, when using equation (\ref{eqn:spec}), the only condition left to check is $\alpha+\beta b + \gamma b \lambda=0$ which is true for the family from Lemma \ref{lem:family}. If $b$ copies of a tight frame is a maximizer of the $q$-th moment among tight frames of the size $bN$, Theorem \ref{thm:p-bound} tells us there is a good chance for one of the representatives of the family to be an optimal projective code.

Various constructions by Bukh and Cox \cite{buk18} are de facto those projective codes from the families described by Lemma \ref{lem:family} minimizing the maximal non-diagonal element of the Gram matrix. One more already known construction arises when the representative of a family is a tight frame itself. This happens when $C$ is a Gram matrix of a tight frame and either $b=1$ or $b>1$ and $\alpha=\frac {b\lambda} {b\lambda-1}$. Then the new tight frame obtained by this construction is the one that is called Gale dual by Cohn, Kumar, and Minton \cite{coh16} or Naimark complement in the frame literature (see, for instance, \cite{cas13}).

\begin{theorem}\label{thm:constr}
Let $C\in K^{M\times M}$ be a Gram matrix of unit representatives of a tight simplex in $K\mathbb{P}^{k-1}$ such that $|C_{ij}|=\beta$ for all $i\neq j$. Then for any $p\neq 0,1$ (including $p=\infty$), there exists a projective code in $K\mathbb{P}^{bm-k-1}$ such that the Gram matrix $A\in K^{bM\times bM}$ of its unit representatives in $K^{bM-k}$ satisfies

$$\left(\sum\limits_{i\neq j} |A_{ij}|^p\right)^{\frac 1 p} = \frac {bM} {\left((bM)^2 \left(\beta^{\frac p {p-1}} + \frac 1 M (1-\beta^{\frac p {p-1}})\right)-bM\right)^{1-\frac 1 p}}.$$
\end{theorem}

\begin{proof}
For the construction we use one of the codes built from the matrix $C$ via Lemma \ref{lem:family}. We just need to choose the right value of the parameter $\alpha$. The way to guess what $\alpha$ should be is to find when the inequality of Theorem \ref{thm:p-bound} becomes the exact equality. The largest eigenvalue of $C$ is, as was mentioned before, $\lambda=\frac M k$. If $A=C\left(\alpha,\frac {-b\lambda + (b\lambda-1)\alpha} {b(1-\lambda)},\frac {b+(1-b)\alpha} {b(1-\lambda)}\right)$, the off-diagonal entries $|A_{ij}|$ are $\alpha-1$ and $\beta \frac {b+(1-b)\alpha} {b(\lambda-1)}$ ($\alpha$ must be at least 1 and no greater than $\frac b {b-1}$). In order for the H\"{o}lder inequality in the proof of Theorem \ref{thm:p-bound} to become the exact equality, the following condition on corresponding entries of the matrices $A$ and $C\otimes J_p$ must hold:

$$\frac {(\alpha-1)^p} {1^q} = \frac {\left(\beta \frac {b+(1-b)\alpha} {b(\lambda-1)}\right)^p} {\beta^q},$$
where $q$ satisfies $\frac 1 p + \frac 1 q=1$.

From this equality we deduce $$\alpha = 1 + \frac 1 {b-1 + \beta^{\frac {q-p} p}b(\lambda-1)}=1 + \frac 1 {b-1 + \beta^{\frac {q-p} p}b(\frac M k-1)}.$$
Let us show that the set constructed for this $\alpha$ satisfies the required condition. The construction for $p=\infty$ can be obtained as a limit so for the remaining part of the proof we use only real $p$.

$$\sum\limits_{i\neq j} |A_{ij}|^p = (b^2-b)M (\alpha-1)^p + ((bM)^2-b^2M) \left(\beta \frac {b+(1-b)\alpha} {b(\lambda-1)}\right)^p$$
Using the condition on $\alpha$ we get

$$\sum\limits_{i\neq j} |A_{ij}|^p= (b^2-b)M (\alpha-1)^p + ((bM)^2-b^2M) (\alpha-1)^p \beta^q=$$

$$=bM \frac {b-1+(bM-b)\beta^q} {\left(b-1 + b(\frac M k-1)\beta^{\frac {q-p} p}\right)^p}.$$
This should be equal to

$$\frac {(bM)^p} {\left((bM)^2 \left(\beta^q + \frac 1 M (1-\beta^q)\right)-bM\right)^{p-1}}=$$

$$=\frac {(bM)^p} {\left(bM(bM-b) \beta^q + bM (b-1)\right)^{p-1}}=$$

$$=\frac {bM} {\left(b-1+(bM-b) \beta^q\right)^{p-1}}.$$
Comparing the two expressions we see that it is sufficient to show that

$$b-1+(bM-b)\beta^q = b-1 + b\left(\frac M k-1\right)\beta^{\frac {q-p} p}.$$
This equality is indeed true because

$$\beta^{q-\frac {q-p} p}=\beta^2=\frac {M-k} {k(M-1)}$$
by Lemma \ref{lem:tight-simplex}.
\end{proof}

The construction of Theorem \ref{thm:constr} works for all tight simplices and almost all values of $p$. In case the tight simplex is also a maximal simplex for this space and $p\geq 2$, the construction of Theorem \ref{thm:constr} satisfies the lower bound of Theorem \ref{thm:p-energy}. We can analyze the situation when this lower bound is attained in more detail.

We know from Theorem \ref{thm:q-measures} that the maximizing isotropic measure is necessarily the uniform measure over a maximal simplex. Assume this simplex is fixed. Then the tight frame in the proof of Theorem \ref{thm:p-bound} is defined uniquely up to unitary transformations and multiplications of each vector by an arbitrary unit number from $K$. Following the proof we can see that $A_{ij} \langle y_j, y_i \rangle$ must necessarily be a negative real number. This essentially means that the construction from Theorem \ref{thm:constr} is required for attaining the bound. The necessary value of a parameter $\alpha$ is explicitly found in the proof of Theorem \ref{thm:constr}. We conclude that the minimizing projective code is uniquely (up to isometries) defined by a maximal simplex in $K\mathbb{P}^{N-d-1}$  and the minimizing matrix from $K^{N\times N}$ of rank $d$ is defined as a Gram matrix of such code, i.e. up to the equivalence relation $x\sim x e$ for any $e\in K^{\times}$ for all unit representatives $x$ of the points of the projective code.

The real maximal simplices in $\mathbb{RP}^1, \mathbb{RP}^2, \mathbb{RP}^6, \mathbb{RP}^{22}$ are known to be unique \cite{sei68, goe75}. There are complex maximal simplices that are known to be unique too, for instance, the one in $\mathbb{CP}^1$ we already discussed in Subsection \ref{subsect:comp} \cite{zau99}. For them, the minimizing codes constructed by Theorem \ref{thm:constr} and satisfying Theorem \ref{thm:p-bound} are unique as well. As an explicit example of such a code we consider the smallest non-trivial real case: projective codes with 6 points in $\mathbb{RP}^3$, i.e $K=\mathbb{R}$ and $d=4$, or, more generally, real $6\times 6$ matrices of rank 4 with the unit diagonal. The Gale dual tight frame consists of 6 points in $\mathbb{R}^2$. By Theorem \ref{thm:q-measures}, the unique maximizer of the $q$-th moment for any $q\in[1,2)$ among all isotropic measures is the uniform distribution over the diagonals of a regular hexagon. The unique maximizer among tight frames of 6 points in $\mathbb{R}^2$ is then a set consisting of two copies of diagonals of a regular hexagon. We can choose the following Gram matrix of its unit representatives:

$$
\begin{bmatrix}
1&1&-\frac 1 2&-\frac 1 2&-\frac 1 2&-\frac 1 2\\
1&1&-\frac 1 2&-\frac 1 2&-\frac 1 2&-\frac 1 2\\
-\frac 1 2&-\frac 1 2&1&1&-\frac 1 2&-\frac 1 2\\
-\frac 1 2&-\frac 1 2&1&1&-\frac 1 2&-\frac 1 2\\
-\frac 1 2&-\frac 1 2&-\frac 1 2&-\frac 1 2&1&1\\
-\frac 1 2&-\frac 1 2&-\frac 1 2&-\frac 1 2&1&1
\end{bmatrix}.
$$

The mimizing matrices from Theorem \ref{thm:constr} will be then defined by

$$
\begin{bmatrix}
1&1-\alpha&1-\frac 1 2\alpha&1-\frac 1 2\alpha&1-\frac 1 2\alpha&1-\frac 1 2\alpha\\
1-\alpha&1&1-\frac 1 2\alpha&1-\frac 1 2\alpha&1-\frac 1 2\alpha&1-\frac 1 2\alpha\\
1-\frac 1 2\alpha&1-\frac 1 2\alpha&1&1-\alpha&1-\frac 1 2\alpha&1-\frac 1 2\alpha\\
1-\frac 1 2\alpha&1-\frac 1 2\alpha&1-\alpha&1&1-\frac 1 2\alpha&1-\frac 1 2\alpha\\
1-\frac 1 2\alpha&1-\frac 1 2\alpha&1-\frac 1 2\alpha&1-\frac 1 2\alpha&1&1-\alpha\\
1-\frac 1 2\alpha&1-\frac 1 2\alpha&1-\frac 1 2\alpha&1-\frac 1 2\alpha&1-\alpha&1
\end{bmatrix},
$$
where $\alpha=1+\frac 1 {1 + 2^{\frac {p-2} {p-1}}}$ provides the one-parametric family of unique minimizing configurations for $p\in(2,+\infty)$. The values of $\alpha$ vary from $\frac 4 3$ to $\frac 3 2$ forming a line segment of minimizers in the vector space of matrices. Projective codes defined by these matrices are unique up to isometries. For each unit representative vector in $\mathbb{R}^4$, there are two opposite choices giving us $2^6$ options. Changing all 6 signs preserves a matrix so there are precisely $2^5=32$ minimizing matrices for each particular $p$.

\section{Discussion}\label{sect:disc}

In this section, we discuss open questions and possible directions of further research in this area.

\begin{enumerate}

\item The natural extension of the linear programming approach for packing and optimality problems is the semidefinite programming approach. Developed initially by Schrijver in the discrete setup \cite{sch05}, it was adapted to packing problems by Bachoc and Vallentin \cite{bac08a} and since then was generalized and successfully used in a variety of discrete geometry and optimization problems \cite{mus14,coh12,del14}. One way to explain semidefinite constraints in this approach is to project points of a spherical code to a unit subsphere of a smaller dimension and use the constraints set by (\ref{eqn:zsf}). A similar approach seems possible for the case of tight frames/isotropic measures. When a tight frame is projected to a subspace, the set of projections form a tight frame in this subspace too, although its vectors should be normalized to suit the restrictions posed on the frame constant. 

\item It seems plausible to extend this approach to packing/optimization problems for Grassmannians. It is possible to extend the notion of Gale transform to linear subspaces (see \cite[p. 133]{eis00}). Tight fusion frames is a natural generalization of tight frames for subspaces \cite{cas04,cas08}. The linear programming machinery is applicable to Grassmannians as well \cite{bac06}. All ingredients of the energy bounds working together for projective codes are present in the Grassmanian case too.    

\item It would be interesting to find out an explanation for computational results in the case of tight frames with 5 points in $\mathbb{R}^3$ as described in Subsection \ref{subsect:comp}. If there is a certain configuration behind this computational bound, it would be interesting to find it. As a counterpart of this question, the case of 5 points on the unit sphere in $\mathbb{R}^3$ is notoriously hard for finding energy minima \cite{sch13, sch16}.

\item The Gale duality of projective codes and tight frames allows us to connect the energies of both objects via H\"{o}lder's inequality. It makes sense to try extending this connection to other energy potentials using the same duality or finding new types of duality with the prospect of finding new energy bounds.

\end{enumerate}

\section{Acknowledgments}
The author would like to thank Bill Martin for initiating this new direction of research on nearly orthogonal vectors, Boris Bukh and Chris Cox for sharing their manuscript and discussing their results, the organizers of the Informal Analysis Seminar at Kent State University where the author learned about these results, Dmitriy Bilyk, Ryan Matzke, Josiah Park, and Oleksandr Vlasiuk for invaluable discussions on a variety of mathematical questions relevant to this research. This material is based upon work supported by the National Science Foundation under Grant No. DMS-1439786 while the author was in residence at the Institute for Computational and Experimental Research in Mathematics in Providence, RI, during the Spring 2018 semester.


\bibliographystyle{amsplain}


\begin{appendices}

\section{Comparing bounds for moments of isotropic measures}\label{app:comp}

In this appendix, we prove Lemma \ref{lem:comp} comparing the bounds of Theorem \ref{thm:yudin} and Corollary \ref{cor:q-measures} for moments of isotropic measures with a fixed number of points.

\comp*

\begin{proof}

When $q=2$, both bounds are equal to $\frac {N^2} d - N$ which is the exact value of $\sum\limits_{i\neq j} |\langle y_1, y_j \rangle|^2$ for any tight frame with the frame constant $\frac N d$. For the remaining part of the proof we assume $q<2$.

It is easy to check that $\beta>\sqrt{\frac {N-d} {d(N-1)}}$ when $N<M$ and $\beta<\sqrt{\frac {N-d} {d(N-1)}}$ when $N>M$. We denote $\sqrt{\frac {N-d} {d(N-1)}}$ by $\alpha$.

Assume $N<M$. Then we know that $\beta>\alpha$. The inequality can be rewritten as

$$N \left( \beta^q + \frac {1-\beta^q} {M} \right) - 1 > (N-1) \alpha^q;$$

$$(N-1)(\beta^q-\alpha^q)+\left( 1 - \frac N M \right)\beta^q >  1 - \frac N M;$$

$$\frac {M(N-1)} {M-N} (\beta^q-\alpha^q)+\beta^q >  1.$$

We denote the left-hand side of this inequality by $\xi(q)$, $q\in[1,2]$. Since $\xi'(q)$ can be written as $c_1 \alpha^q + c_2 \beta^q = \beta^q (c_1 (\frac \alpha \beta)^q + c_2)$ and $c_1\neq 0$, $\xi'(q)$ cannot possibly have more than one zero on $[1,2]$. We know that $\xi(2)=1$ so if we can show that $\xi(1)>1$ and $\xi'(2)<0$, then we can claim that $\xi(q)>1$ for $q\in[1,2)$. Otherwise, $\xi'(q)$ would have at least two zeros.

For the first part, we want to prove

$$\frac {M(N-1)} {M-N} (\beta-\alpha)+\beta >  1.$$

We rewrite this inequality as

$$\frac {M-1} {M} \beta + \frac 1 M > \frac {N-1} {N} \alpha + \frac 1 N.$$

Introducing the function $\phi(t) = \frac {t-1} t \sqrt{\frac {t-d} {d(t-1)}} + \frac 1 t$ for all $t\geq d$, we can see that the inequality is equivalent to $\phi(M)>\phi(N)$. It is easy to check that $\phi$ is a strictly increasing function so this inequality must hold.

For the second part, we calculate the derivative of $\xi$ at 2.

$$\xi'(2)=\frac {M(N-1)} {M-N} (\beta^2 \ln \beta- \alpha^2\ln\alpha)+ \beta^2\ln\beta = $$

$$= \ln\beta \left(\frac {M(N-1)} {M-N} (\alpha^2-\beta^2)+\beta^2\right) + (\ln\alpha - \ln\beta) \frac {M(N-1)} {M-N} \alpha^2.$$

Then we use $\xi(2)=1$.

$$\xi'(2) = \ln\beta + (\ln\alpha - \ln\beta) \frac {M(N-1)} {M-N} \alpha^2 = $$

$$= \ln\beta + (\ln\alpha - \ln\beta) \frac {M(N-1)} {M-N} \frac {N-d} {d(N-1)} = \ln\beta + (\ln\alpha - \ln\beta) \frac {M(N-d)} {(M-N)d} = $$

$$=\ln\beta \frac {N(M-d)} {(M-N)d} - \ln\alpha \frac {M(N-d)} {(M-N)d} = \frac {MN} {(M-N)d} \left(\ln\beta \frac {M-d} M - \ln\alpha \frac {N-d} {N}\right).$$

Here we introduce the function $\psi(t)=\frac {t-d} t \ln \sqrt{\frac {t-d} {d(t-1)}}$ for $t>d$ and see that $\xi'(2)<0$ if $\psi(M)<\psi(N)$. It is easy to check that $\psi$ is decreasing for $t>d$ so the required inequality must hold.

The case $N>M$ can be done the same way with the corresponding changes in signs of the inequalities.

%

\end{proof}

\section{Discrete isotropic measures with orthogonal vectors}\label{app:orth}

In this appendix, we prove that discrete isotropic measures with a pair of orthogonal vectors cannot maximize the $q$-th moments for $q\in[1,2)$.

\orth*

\begin{proof}
We consider a discrete isotropic measure with a pair of orthogonal vectors and provide a perturbation of this measure whose $q$-th moment is larger.

Assume $supp(\mu)=\{u,u_1,\ldots,u_N\}$ and $u$ is orthogonal to $u_1$. Let $\mu(u)=p$ and $\mu(u_i)=p_i$ for all $i$. In a subspace orthogonal to $u$ we choose a regular simplex $\{v_1,v_2,\ldots,v_d\}$ such that $|v_i|=1$ for all $i$ and $v_1=\frac {u_1} {|u_1|}$. We fix a small parameter $\delta>0$ and define the perturbed measure $\mu'$ as follows. $supp(\mu')=\{v_1',\ldots, v_d',u_1',\ldots,u_N'\}$, where $v_i'=\alpha u +\delta v_i$ for a fixed real constant $\alpha$ and for all $i$, $1\leq i\leq d$, and $u_i'=\beta u_i$ for a fixed real constant $\beta$ and for all $i$, $1\leq i\leq N$. $\mu'(u_i')=p_i$ for all $i$ and $\mu'(v_i')=\frac 1 d p$. Constants $\alpha$ and $\beta$ depend on $\delta$ and are chosen in such a manner that $\mu'$ is isotropic.

We will show that $\alpha=\sqrt{1+\frac {1} {d-1} \left(\frac 1 {|u|^2} - pd\right)\delta^2}$ and $\beta^2=\sqrt{1-\frac {pd } {d-1}\delta^2}$ satisfy the required condition. For an arbitrary $w\in K^d$ we represent it as a sum of orthogonal vectors $w_u+w_v$, where $w_u=u\gamma$ for some $\gamma\in K$ and $w_v$ is orthogonal to $u$. Using $|\langle u,w \rangle|^2 = |u|^2 |w_u|^2$ and the isotropy of the measure $\mu$ we get

\begin{equation}\label{eqn:rem_vect}
\sum\limits_{i=1}^N p_i |\langle u_i, w \rangle|^2 = \frac 1 d |w|^2 - p |\langle u,w \rangle|^2 = \frac 1 d |w|^2 - p|u|^2 |w_u|^2 = \left(\frac 1 d - p|u|^2\right) |w_u|^2 + \frac 1 d |w_v|^2.
\end{equation}

$$\sum\limits_{i=1}^d |\langle v_i', w \rangle|^2 = \sum\limits_{i=1}^d |\langle \alpha u + \delta v_i, w_u+w_v \rangle|^2 =\sum\limits_{i=1}^d |\alpha\langle u, w_u\rangle + \delta \langle v_i, w_v \rangle |^2=$$

$$=\sum\limits_{i=1}^d \overline{\alpha\langle u, w_u\rangle + \delta \langle v_i, w_v \rangle} (\alpha\langle u, w_u\rangle + \delta \langle v_i, w_v \rangle) =$$

$$=\sum\limits_{i=1}^d (\alpha^2 |\langle u,w_u \rangle|^2 + \alpha\delta \overline{\langle u, w_u\rangle}  \langle v_i, w_v \rangle +  \alpha\delta \langle u, w_u\rangle \overline{\langle v_i, w_v \rangle} + \delta^2 |\langle v_i,w_v \rangle|^2)=$$

$$=d\alpha^2 |u|^2 |w_u|^2 + \alpha\delta \overline{\langle u, w_u\rangle}  \langle \sum\limits_{i=1}^d v_i, w_v \rangle + \alpha\delta \langle u, w_u\rangle  \overline{\langle \sum\limits_{i=1}^d v_i, w_v \rangle} +\delta^2\sum\limits_{i=1}^d|\langle v_i,w_v \rangle|^2$$

Since $\sum_{i=1}^d v_i=0$, we can simplify this sum to

$$\sum\limits_{i=1}^d |\langle v_i', w \rangle|^2=d\alpha^2 |u|^2 |w_u|^2 +\delta^2\sum\limits_{i=1}^d|\langle v_i,w_v \rangle|^2.$$

Since the vectors $v_i$ form a regular simplex in a $(d-1)$-dimensional space, they represent a tight frame with the frame constant $\frac d {d-1}$. $w_v$ also belongs to this $(d-1)$-dimensional subspace so $\sum_{i=1}^d|\langle v_i,w_v \rangle|^2 = \frac d {d-1} |w_v|^2$. This implies the following equation.

\begin{equation}\label{eqn:pert}
\sum\limits_{i=1}^d |\langle v_i', w \rangle|^2=d\alpha^2 |u|^2 |w_u|^2 +\delta^2\frac d {d-1} |w_v|^2.
\end{equation}

Combining equations (\ref{eqn:rem_vect}) and (\ref{eqn:pert}) we conclude

$$\sum\limits_{i=1}^N p_i |\langle u_i', w \rangle|^2 + \sum\limits_{i=1}^d \frac p d |\langle v_i',w \rangle|^2=$$

$$=\beta^2\left(\left(\frac 1 d - p|u|^2\right) |w_u|^2 + \frac 1 d |w_v|^2\right) + \frac p d \left(d\alpha^2 |u|^2 |w_u|^2 +\delta^2\frac d {d-1} |w_v|^2\right)=$$

$$= \left(\beta^2\left(\frac 1 d - p|u|^2\right) + p\alpha^2|u|^2 \right) |w_u|^2 + \left(\frac {\beta^2} {d} + \frac {\delta^2 p} {d-1}\right) |w_v|^2=$$

$$=\frac 1 d |w_u|^2 + \frac 1 d |w_v|^2 = \frac 1 d |w|^2$$
so $\mu'$ is indeed isotropic.

Now we want to analyze how the $q$-th moment changes under the perturbation. There are three types of scalar profucts to check: $\langle u_i,u_j \rangle$, $\langle u,u \rangle$, and $\langle u_i,u \rangle$.

\begin{enumerate}

\item Since $\beta=1+O(\delta^2)$, all changes $p_i p_j|\langle u_i', u_j'\rangle|^q - p_ip_j|\langle u_i, u_j\rangle|^q$ are $O(\delta^2)$.

\item $|\langle v_i', v_j' \rangle|^q=|\alpha^2|u|^2 + \delta^2 \langle v_i, v_j \rangle|^q=|u|^{2q}+O(\delta^2).$ Summing over all pairs of $v_i', v_j'$ we get that

$$\sum\limits_{i,j=1}^d\frac {p^2} {d^2} |\langle v_i', v_j' \rangle|^q-p^2|\langle u,u \rangle|^q$$
is $O(\delta^2)$ as well.

\item First, we show that for $\langle u_k,u\rangle \neq 0$, the changes are $O(\delta^2)$ too.

$$\sum\limits_{i=1}^d |\langle u_k', v_i'\rangle|^q = \beta\sum\limits_{i=1}^d |\alpha \langle u_k, u \rangle + \delta \langle u_k,v_i\rangle|^q =\sum\limits_{i=1}^d |\alpha \langle u_k, u \rangle + \delta \langle u_k,v_i\rangle|^q + O(\delta^2) = $$

$$=\sum\limits_{i=1}^d \left(\overline{\alpha \langle u_k, u \rangle + \delta \langle u_k,v_i\rangle} (\alpha \langle u_k, u \rangle + \delta \langle u_k,v_i\rangle) \right)^{\frac q 2} +O(\delta^2)=$$

$$=\sum\limits_{i=1}^d \left(\alpha^2 |\langle u_k, u \rangle|^2 + \alpha\delta \overline{\langle u_k,u\rangle} \langle u_k,v_i\rangle + \alpha\delta \langle u_k,u\rangle \overline{\langle u_k,v_i\rangle} + \delta^2 |\langle u_k,v_i\rangle|^2) \right)^{\frac q 2} + O(\delta^2)=$$

$$=d \alpha^q |\langle u_k, u \rangle|^q + \frac q 2 \alpha\delta  (\alpha^2 |\langle u_k, u \rangle|^2)^{\frac q 2 -1} \sum\limits_{i=1}^d (\overline{\langle u_k,u\rangle} \langle u_k,v_i\rangle + \langle u_k,u\rangle \overline{\langle u_k,v_i\rangle}) + O(\delta^2)=$$

$$=d \alpha^q |\langle u_k, u \rangle|^q + \frac q 2 \alpha^{q-1} \delta  |\langle u_k, u \rangle|^{q -2}\left( \overline{\langle u_k,u\rangle} \langle u_k, \sum\limits_{i=1}^d v_i\rangle + \langle u_k,u\rangle \overline{\langle u_k, \sum\limits_{i=1}^d  v_i\rangle}\right) + O(\delta^2)=$$

$$=d \alpha^q |\langle u_k, u \rangle|^q + O(\delta^2) = d |\langle u_k, u \rangle|^q + O(\delta^2) $$

Therefore, $\sum\limits_{i=1}^d  p_k \frac p d |\langle u_k',v_i'\rangle|^q - p_k p |\langle u_k, u \rangle|^q$ is $O(\delta^2)$ too.

Finally, for any $k$ such that $\langle u_k,u\rangle = 0$, $\sum\limits_{i=1}^d  p_k \frac p d |\langle u_k',v_i'\rangle|^q - p_k p |\langle u_k, u \rangle|^q \geq 0$. Moreover, for $k=1$, we know that

$$\sum\limits_{i=1}^d  p_1 \frac p d |\langle u_1',v_i'\rangle|^q - p_1 p |\langle u_1, u \rangle|^q \geq p_1 \frac p d |\langle u_1',v_1'\rangle|^q =$$

$$=p_1 \frac p d \left|\langle \beta u_1, \alpha u + \delta \frac {u_1} {|u_1|} \rangle\right|^q = p_1 \frac p d \beta^q |u_1|^q \delta^q = \Omega(\delta^q).$$

This means that the increase caused by the perturbed orthogonal pair $(u, u_1)$ is at least $\Omega(\delta^q)$ and all possible decreases were $O(\delta^2)$. Since $q<2$, for a sufficiently small $\delta$, the perturbed isotropic measure $\mu'$ has a larger $q$-th moment than the initial measure $\mu$.

\end{enumerate}

\end{proof}

\end{appendices}

\end{document}